\documentclass[11pt]{amsart}
\usepackage{amsmath,amssymb,amsthm,mathtools,dsfont,bm}

\usepackage{slashed}
\usepackage{mathtools}
\usepackage{esint}
\mathtoolsset{showonlyrefs}

\usepackage[dvips]{graphicx}
\usepackage{color}
\usepackage{wasysym}
\newtheorem{lem}{Lemma}[section]
\newtheorem{prop}{Proposition}[section]
\newtheorem{cor}{Corollary}[section]

\newtheorem{thm}{Theorem}[section]

\theoremstyle{definition}

\theoremstyle{remark}

\theoremstyle{remark}
\newtheorem{remark}{Remark}[section]

\numberwithin{equation}{section}

\newcommand{\vertiii}[1]{{\left\vert\kern-0.25ex\left\vert\kern-0.25ex\left\vert #1
    \right\vert\kern-0.25ex\right\vert\kern-0.25ex\right\vert}}

\definecolor{blu}{rgb}{0,0,1}

\newcommand{\be}{\begin{equation}}
\newcommand{\ee}{\end{equation}}

\catcode`@=11
\def\section{\@startsection{section}{1}%
  \z@{1.5\linespacing\@plus\linespacing}{.5\linespacing}%
  {\normalfont\bfseries\large\centering}}
\catcode`@=12

\begin{document}
\title[On completeness of modified wave operators  ]{On completeness of modified wave operators for defocusing NLS }
\date{}
\author{Vladimir Georgiev}
\address{Vladimir Georgiev, Dipartimento di Matematica ,  Universit\`a di Pisa \\ Largo B. Pontecorvo 5, 56100 Pisa, Italy and
\\Faculty of Science and Engineering \\ Waseda University \\
 3-4-1, Okubo, Shinjuku-ku, Tokyo 169-8555 \\
Japan and
\\Institute of Mathematics and Informatics, Bulgarian Academy of Sciences, Acad.
Georgi Bonchev Str., Block 8, 1113 Sofia, Bulgaria}

\author{Tohru Ozawa}
\address{Tohru Ozawa, Department of Applied Physics \\ Waseda University \\
 3-4-1, Okubo, Shinjuku-ku, Tokyo 169-8555 \\
Japan }

\begin{abstract}

In this manuscript, we study modified scattering for the nonlinear defocusing Schr\"odinger equation with a critical gauge-invariant nonlinearity of order $1+2/n.$

We address the following question: Given initial data $u_0$ in an appropriate weighted Sobolev space, what is the leading term in the asymptotic behavior of the solution $u(t)$ as $t \to \pm \infty$? More precisely, we seek a final state $u^+$ as $t \to + \infty$ (in a space of type similar to the space of $u_0$) such that
\begin{equation} \label{d1}
\begin{aligned}
& \left\|u(t)- e^{i \Phi^{+}(t)} U(t)u^{+}\right\|_{L^2}=0  \ \ \mbox{as $t \to +\infty$}.
\end{aligned}
    \end{equation}

    This result provides the modified profile of the leading term, where the linear propagator $U(t)$ for the linear Schr\"odinger equation is perturbed by a modified phase function $\Phi^{+}(t).$

    The solution to this problem can be reformulated in terms of the completeness of wave operators. For $n=1$, we obtain affirmative  answer, provided appropriate control on $L^\infty$ norm even for large initial data. For $n = 2$ completeness is established under suitable control of the $L^\infty$ norm of the solution $u(t).$

In the case where no control on the $L^\infty$ norm is available, we establish similar asymptotic behavior by replacing the convergence in  $L^2$ norm in \eqref{d1} with  a weaker convergence.

\end{abstract}

\maketitle


{Subject Class:}\ \ {35Q55 \and 35P25 \and 35B40}

\section{Introduction}
Consider the Cauchy problem for the defocusing NLS
\begin{equation}\label{eq.NLSf1}
    \begin{aligned}
    & (i \partial_t + \frac{1}{2}\Delta_x)u = u |u|^{2/n}, \ \ t \geq 0, \\
    & u(0)= u_0 \in \mathcal{H}^{s, \sigma} = H^s(\mathbb{R}^n) \cap \mathcal{F}(H^\sigma(\mathbb{R}^n)),
    \end{aligned}
\end{equation}
where $ u : \mathbb{R} \times \mathbb{R}^n \ni (t,x) \mapsto u(t,x) \in \mathbb{C},$ $\partial_t = \partial/\partial t,$
$\Delta = \partial_1^2 + \cdots +\partial_n^2,$ $ \partial_j = \partial / \partial x_j,$ $H^s$ is the classical Sobolev space, see \eqref{eq.clsob1},
and where $\mathcal{F}$ is the Fourier transform introduced in \eqref{eq.ftr1}.

We are interested in the long-range effects of \eqref{eq.NLSf1}.

The main questions concerning modified scattering are described below.

\begin{enumerate}
    \item[Q1)]
     Given $u^+$ in a space of type $\mathcal{H}^{s,\sigma}$, show that there exists $u_0$ and a global solution $u$ to the NLS so that
    \begin{equation}
\begin{aligned}
& \lim_{t \to +\infty} \left\| u(t)- e^{i \Phi^{+}(t)} U(t)u^{+} \right\|_{L^2}= 0 .
\end{aligned}
    \end{equation}
    Here $ U(t) = e^{i\frac{t}{2} \Delta }$ is the free propagator, while $\Phi^{+}(t)=\Phi^+ (t,x)$ is appropriate real -valued function depending on the final state $u^+.$
    This gives the modified wave operator
\begin{equation}\label{eq.mvo1}
    W_+ (u^+) = u_0 = \lim_{t \to +\infty} U(-t)e^{-i \Phi^{+}(t)}u^+,
\end{equation}
where the limit is taken in $L^2$ sense. Similarly, we can ask if modified wave operator
\begin{equation} \label{eq.mvo2}
    W_- (u^-) = u_0  ,
\end{equation}
such that
 \begin{equation}
\begin{aligned}
& \lim_{t \to -\infty} \left\| u(t)- e^{i \Phi^{-}(t)} U(t)u^{-} \right\|_{L^2}= 0
\end{aligned}
    \end{equation}
exists?

\item[Q2)] Another key question is the completeness of wave operators.
Given $u_0$ in space of type $\mathcal{H}^{\sigma,s},$ suppose that there is a global mild solution $u \in C([0,+\infty);\mathcal{H}^{\sigma,s}). $ One can see for example \cite{O91} for the case $n=1$ and \cite{HN06} for $n=1,2,3$ for the case of initial data in $\mathcal{H}^{0,n/2+}$ and appropriate smallness assumption. Then we are looking for final state $u^+$ as $t \to + \infty$ (in a space of type similar to the space of $u_0$) such that
\begin{equation} \label{decr1}
\begin{aligned}
& \lim_{t \to +\infty}\left\|u(t)- e^{i \Phi^{+}(t)} U(t)u^{+}\right\|_{L^2}=0 .
\end{aligned}
    \end{equation}

In this way we can define
\begin{equation}
    (W_+)^{-1}: u_0 \mapsto u^+
\end{equation}
and in a similar way
\begin{equation}
    (W_-)^{-1}: u_0 \mapsto u^-.
\end{equation}
Further we can introduce the scattering operator
    \begin{equation}
        S =  (W_-)^{-1}W_+: u^+ \mapsto u^-.
    \end{equation}
\end{enumerate}

The defocussing NLS has conservation quantities: the mass
\begin{equation}
    m(u(t)) = \frac{1}{2}\int_{\mathbb{R}^n} |u(t,x)|^2 dx
\end{equation}
and the energy
\begin{equation}
    E(u(t)) = \frac{1}{2}\int_{\mathbb{R}^n} |\nabla u(t,x)|^2 dx + \frac{n}{2n+2} \int_{\mathbb{R}^n} |u(t,x)|^{2/n+2} dx.
\end{equation}

\subsection{Known results}

Recall that more general NLS
\begin{equation}\label{eq.nls93}
    \begin{aligned}
     (i \partial_t + \frac{1}{2}\Delta)u = u(t)|u(t)|^{p-1}, \ \ t \in [0,+\infty)
    \end{aligned}
\end{equation}
is studied intensively in the recent decades. The critical Strauss exponent
\begin{equation}
    \gamma(n) = \frac{(n+2) + \sqrt{n^2+ 12n + 4}}{2n}
\end{equation}
plays crucial role in the proof of existence and completeness of the wave operators. Note that $\gamma(n) > 1+2/n$ so this case corresponds to the short range regime.

The case $p>\gamma(n)$ and large initial data in
\begin{equation}
 \mathcal{H}^{1,1}=   H^1(\mathbb{R}^n) \cap \mathcal{F}(H^1(\mathbb{R}^n))
\end{equation}
is studied in \cite{HT86}, \cite{T85}   and existence and completeness of wave operators is obtained. The nonlinearity with Strauss critical exponent is studied in \cite{C03}, \cite{CW92}, \cite{NO01}, \cite{NONO02}.  The crucial case $p=\gamma(n)$ is studied in \cite{NONO02}.
The precise asymptotics for the case $p>\gamma(n)$ is discussed in \cite{KO05}.

Let us mention that most of the  known  results are treating modified scattering for defocussing critical NLS with $p=1+2/n$ and they  need a kind of smallness assumption.

For the special case $n=1$ it is known from the results in \cite{DZ03} that we have
the $L^\infty$ bound
\begin{equation}
    \|u(t) -  t^{-1/2} \alpha(x) e^{ix^2/4t - i \nu(x)  \log (2 t)}    \|_{L^\infty} \leq C (1+t)^{-1/2-\kappa}, \ \kappa \in (0,1/4),
\end{equation}
where $\alpha \in \mathcal{H}^{1,1}$ is complex - valued and  $\nu(x) >0.$ The result is obtained by using the inverse scattering method working in dimension $n=1$ and exploiting the reflection coefficient $r(x)$ that is assumed to be in the space
\begin{equation}
    \mathcal{H}^{1,1} \cap \left\{r ; \|r\|_{L^\infty} < 1 \right\}.
\end{equation}

This result shows that  one can have solution satisfying the decay property
\begin{equation}\label{Linfc1}
    \|u(t,\cdot)\|_{L^\infty} \lesssim t^{-1/2},
\end{equation}
when  the smallness condition is expressed in terms of explicit bound on the reflection coefficient.

The work \cite{O91} is  the first case when modified wave operators \eqref{eq.mvo1}, \eqref{eq.mvo2} for $ n=1$
are defined even in the case where the right hand side is replaced by
\begin{equation}\label{eq.NLS2}
    \begin{aligned}
    f(|u(t)|^2) = \lambda u(t) |u(t)|^{2/n} + \mu u(t)|u(t)|^{p-1} , \ \ t \in [0,+\infty), p> 1+2/n; \\
    \end{aligned}
\end{equation}
with $\lambda \in \mathbb{R} \backslash\{0\}$ and $\mu \in \mathbb{R}$.  The modified phase function $\Phi^{ \pm}$ and $\Phi_0^{ \pm}$ are defined by
\begin{equation}
\begin{aligned}
& \Phi^{ \pm}(t, x) =\mp \lambda \log |t|\left|\widehat{u^{ \pm}}\left(t^{-1} x\right)\right|^2  \pm
\frac{2\mu}{p-3}
|t|^{-(p-3) / 2}\left|\widehat{u^{ \pm}}\left(t^{-1} x\right)\right|^{p-1} ,\\
& \Phi_0^{ \pm}(t, x)=\mp \lambda \log |t|\left|\widehat{u^{ \pm}}\left(t^{-1} x\right)\right|^2.
\end{aligned}
\end{equation}
The main result in \cite{O91} implies \eqref{eq.mvo1} and \eqref{eq.mvo2} assuming $u^{\pm} \in \mathcal{H}^{1,3}$ with $\left\|\widehat{u^{\pm}}\right\|_{L^\infty}<\varepsilon$. The decay rate in \eqref{decr1} is deduced in \cite{O91} by verifying the decay estimate
\begin{equation} \label{decr2}
\begin{aligned}
& \left\|u(t)- e^{i \Phi^{+}(t)} U(t)u^{+}\right\|_{\mathcal{H}^{1,0}}\lesssim t^{-1+}  \ \ \mbox{as $t \to +\infty$}.
\end{aligned}
    \end{equation}
Later  on the decay result in \cite{O91} has been improved in \cite{C2001} by using the same assumptions on $u^+.$ The cases $n =1,2,3$ are studied in  \cite{GO93}, where the  Dollard decomposition is used and the existence of wave operators is established.
Concerning the Dollard decomposition and modified wave operators we can refer also to \cite{MS16}.

Further, the work \cite{GO93} treats the Cauchy problem for small and sufficiently regular asymptotic states, i.e.  the initial data at $t\to +\infty$ and then the modified scattering is established.

The case of small initial data in the space
$\mathcal{H}^{1,1}$
and $p = 1+2/n$ is treated in   the classical work \cite{HN98AJM}, where the question Q2) is studied and  the  modified scattering is represented by the following asymptotics as $t\to +\infty$
\begin{equation}
    u(t) = (it)^{-n/2} w\left(\frac{x}{t} \right)  e^{i x^2/2t + |w(x/t)|^2 \log t + i \Phi(x/t)}  + O(t^{-n/2-\delta}),
\end{equation}
where $w,\Phi $ are $L^\infty$ functions. Alternative proof for the case $n=1$ can be found in \cite{KP11}.

The work \cite{HN06} deals with existence and completeness of the wave operators for $n=1,2,3$ having less regularity of the initial data $u_0, u^\pm  \in \mathcal{H}^{0,n/2+}.$
More precisely, if $u_{+} \in \mathcal{H}^{0, \alpha}$ and $\left\|\widehat{u_{+}}\right\|_{{L}^{\infty}}=\varepsilon$, where $\varepsilon$ is sufficiently small and $\frac{n}{2}<\alpha<\min \left(n, 2,1+\frac{2}{n}\right)$, then there exists a unique global solution $u$ to the NLS satisfying
\begin{equation}
    \begin{aligned}
        \left\|u(t)-\frac{e^{\frac{i | x|^2}{2 t}}}{(i t)^{\frac{n}{2}} } \widehat{u_{+}}\left(\frac{\cdot}{t}\right) \exp \left(-i \lambda \left|\widehat{u_{+}}\left(\frac{\cdot}{t}\right)\right|^{\frac{2}{n}} \log t\right)\right\|_{\mathcal{H}^{\delta,0}} \leq \frac{C}{ t^{\frac{\beta-\delta}{2}+\mu}},
    \end{aligned}
\end{equation}
where $\frac{n}{2}<\beta<\alpha$
and  $0 \leq \delta \leq \beta, \mu>0$.

Further the operator $W_{-}^{-1}$ is constructed in \cite{HN06}, so that
$$
W_{-}^{-1}: u_0 \in \mathcal{H}^{0, \beta} \mapsto u_{-} \in \mathcal{H}^{0, \delta}
$$
where $\frac{n}{2}<\delta<\beta<\min \left(2,1+\frac{2}{n}\right)$. Therefore we have the modified scattering operator
$$
S_{+}=W_{-}^{-1} W_{+}: \mathcal{H}^{0, \alpha} \mapsto \mathcal{H}^{0, \delta},
$$
where $\frac{n}{2}<\delta<\alpha<\min \left(n, 2,1+\frac{2}{n}\right)$, provided that the norm $\left\|u_{+}\right\|_{\mathcal{H}^{0, \alpha}}$ is sufficiently small.

The work \cite{LS06} gives another interesting result for 1d long range scattering of NLS.
The authors  make the  ansatz
$$
u(t, x)=t^{-1 / 2} \mathrm{e}^{\mathrm{i} x^2 / 4 t} V(t, y), \quad s=t, \quad y=x / t
$$
Then the NLS problem is transformed into

\begin{equation}
    \mathrm{i} \partial_s V-\beta s^{-1}|V|^2 V + s^{-2} \partial_y^2 V=0
\end{equation}
and then $V$ is represented as $V = a e^{i\phi}.$
Then the  authors prove the existence and  completeness of wave operators under smallness assumption of the smooth initial data.

In \cite{MVH22} modified 1d scattering is established under smallness assumptions in $\mathcal{H}^{1,1}.$

The asymptotic expansion with respect to the initial small data size  is given in \cite{C24}.

On the other hand, if initial data are in $H^1$ only, then  the following estimate is known (see \cite{PV09}, \cite{CGT09}) for the 1d case
\begin{equation}
    \|u\|_{L^6_{(1,+\infty)}L^6_x}^2 + \left\| \nabla_x (|u|^2) \right\|_{L^2_{(1,+\infty)}L^2_x} \lesssim \sup_{t \geq 1}\|u(t)\|_{L^2} \|u(t)\|_{\dot{H}^{1/2}}.
\end{equation}

Using  smallness assumption
\begin{equation} \label{eq.sm11}
    \left\|u_0 \right\|_{\mathcal{H}^{0,1}} \leq \varepsilon << 1,
\end{equation}
Ifrim and Tataru in \cite{IT24} have obtained the smallness of $L^\infty $ norm and the expected decay $t^{-1/2}.$ Note that the smallness assumption \eqref{eq.sm11} can be weaken to  the smallness assumption of the $\mathcal{H}^{0,1/2+}$ and this is exactly the case treated in  \cite{HN06}.  The decay of $L^\infty$ norm is studied in \cite{HT85} under the assumption that $p \geq 1+4/n.$
In the case $n=1$ the small weighted Sobolev norms on initial data (actually exponential decay in $x$ is assumed) the modified scattering is deduced in \cite{HKN98}.

In the case of more general cubic nonlinearities and $n=1$ the existence of wave operators are established in \cite{HN09} by requiring
$u_0 \in \mathcal{H}^{2,2}.$

\subsection{Purpose of the work}
The above short overview on existing results shows that most of the results require smallness of the initial data

Our goal is  to discuss Q2) without smallness on the initial data for $n=1,2$ or $n\geq 3$
and assuming $L^\infty $ control of the solution.

\begin{remark}
    Let us mention that in the case $n=1$ the question Q1) is treated in a  preprint \cite{KM23}.
\end{remark}

\section{Main results}

\begin{thm}\label{MT1} Assume $n \leq 2$ and let $u(t)$ be a global solution to \eqref{eq.NLSf1} satisfying
    \begin{equation}\label{linf5}
        \|u(t,\cdot)\|_{L^\infty} \lesssim t^{-n/2}, \ \ t \geq 1
    \end{equation}
    with  $u(0)=u_0 \in \mathcal{F}(H^1).$ Then there exists a unique $u^+ \in L^2$ so that
    \begin{equation}\label{mwoe60}
        \lim_{t \to + \infty} \|  u(t) - e^{-i \Phi_t(u)} U(t)u^+ \|_{L^2} = 0,
    \end{equation}
    where
    \begin{equation}\label{eq.phd}
     \Phi_t(u) (x) = \int_1^t |u(\tau, \tau x/t)|^{2/n} d \tau.
    \end{equation}
\end{thm}

\begin{remark}
    From the results, treating 1d case  in \cite{KM25}, we know that the modified wave operators
    \begin{equation}
        \begin{aligned}
            W_+:  u^+ \in \mathcal{F}(H^{1+\varepsilon})   \mapsto u_0 = \lim_{t \to +\infty} U(-t)e^{-i \Phi^{+}(t)}u^+ \in \mathcal{F}(H^{1}), \\
            W_-:  u^- \in \mathcal{F}(H^{1+\varepsilon})   \mapsto u_0 = \lim_{t \to +\infty} U(-t)e^{-i \Phi^{-}(t)}u^- \in \mathcal{F}(H^{1})\\
        \end{aligned}
    \end{equation}
    are well defined and any $u_0 \in \mathrm{Ran} (W_-)$ satisfies the decay assumption \eqref{linf5}. From this fact and from Theorem \ref{MT1} we can conclude that
    \begin{itemize}
        \item the modified  wave operators are complete,

        \item the scattering operator $S = (W_+)^{-1}W_-$ is well - defined.
    \end{itemize}

\end{remark}

\begin{thm}\label{MT2} Under the assumptions of Theorem \ref{MT1} if in addition $u_0 \in H^1,$ then $u^+ \in H^1$ and for any $\varphi \in H^1$ we have
    \begin{equation}\label{mwoe60m}
        \lim_{t \to + \infty} \left( \nabla \left( U(-t)e^{i\Phi_t(u)}u(t)  -u^+\right) | \nabla \varphi \right) = 0,
    \end{equation}
    where $(\cdot | \cdot)$ is  the inner product in $L^2$ defined by \eqref{ip1}.
\end{thm}
\begin{remark}
    The limit in \eqref{mwoe60m} and  \eqref{mwoe60} show that the weak limit in   $H^1$ of
$$ U(-t)e^{i\Phi_t(u)}u(t)  -u^+  $$
is $0.$

\end{remark}

Let us give an outline of the proofs of these Theorems. We use the Dollard decomposition (see \cite{GO93}, \cite{MS16})
of  the Schr\"odinger group $U(t)=e^{i\frac{t}{2}\Delta }$ represented by
\begin{equation}
   (U(t) f)(x)= \int_{\mathbb{R}^n} K_t(x-y) f(y)dy,
\end{equation}
where
\begin{equation}
  K_t(x) = (2\pi i t)^{-n/2} e^{ix^2/(2t)}, t\neq 0.
\end{equation}

The Dollard decomposition has the form
\begin{equation}\label{eq.doll5aa}
\begin{aligned}
   & U(t) = M(t) D(t) \mathcal{F} M(t),
\end{aligned}
\end{equation}
where
\begin{equation}\label{eq.MD1}
\begin{aligned}
   & M(t) (f)(x) = e^{ix^2/(2t)} f(x), \ \ M(t)^{-1} = M(-t), \\
    &D(t)\, (f) (x) = (it)^{-n/2} f \left(\frac{x}{t} \right), \ D(t)^{-1}\, (f)(x) = (it)^{n/2} f \left(x t \right).
\end{aligned}
\end{equation}
The proofs then rely on the decay assumption and the pseudoconformal energy. To establish estimates for the pseudoconformal energy, we apply the pseudoconformal transformation.
\begin{equation}\label{eq.PCT11}
    u(t,x) = (1+t)^{-n/2} v\left(\frac{t}{1+t},\frac{x}{1+t} \right)  e^{\frac{ix^2}{2(1+t)}},
\end{equation}
well - defined for $t\geq 0, x\in  \mathbb{R}^n.$ Then $v \in C([0,1);\mathcal{H}^{1,1}) $  is a solution to the Cauchy problem
\begin{equation}\label{eq.PSCv1}
    \begin{aligned}
    & (i \partial_t + \frac{1}{2}\Delta_x)v = \frac{v |v|^{2/n}}{1-t} , \ \ t \in [ 0, 1); \\
    & v(0)= u_0 \in \mathcal{H}^{1,1}.
    \end{aligned}
\end{equation}

In the case $n=1$ we can control the $L^\infty$ norm and this is our next result.

\begin{thm}\label{T.1}
    For any $R>0$ there exists $C(R)$ so that for any initial data satisfying
    \begin{equation}
        \begin{aligned}
         &   u_0 \in \mathcal{H}^{1,1} , \ \ \|u_0\|_{\mathcal{H}^{1,1} } \leq R
        \end{aligned}
    \end{equation}
    the Cauchy problem
    \begin{equation}\label{eq.NLSn1}
    \begin{aligned}
    & (i \partial_t + \frac{1}{2}\Delta_x)u = u |u|^2, \ \ t \geq 0; \\
    & u(0)= u_0 \in  \mathcal{H}^{1,1},
    \end{aligned}
\end{equation}
    has a unique global solution $u \in C([0,+\infty);\mathcal{H}^{1,1})$ such that
    \begin{equation}\label{13}
        \|u(t,\cdot)\|_{L^\infty}  \leq C(R)\ t^{-1/3}, \ \ t \geq 1.
    \end{equation}
\end{thm}

In the case $n=2$ we can control the $L^\infty$ norm under the assumption that initial data are radial.

\begin{thm}\label{T.2}
    For $n=2$ any $R>0$ there exists $C(R)$ so that for any initial data satisfying
    \begin{equation}
        \begin{aligned}\label{rad}
         &   u_0 \in \mathcal{H}^{1,1}_{rad}, \  \|u_0\|_{\mathcal{H}^{1,1}_{rad}}  \leq R
        \end{aligned}
    \end{equation}
    the Cauchy problem
    \begin{equation}\label{eq.NLSn2}
    \begin{aligned}
    & (i \partial_t + \frac{1}{2}\Delta_x)u = u |u|, \ \ t \geq 0; \\
    & u(0)= u_0 \in   \mathcal{H}^{1,1}_{rad},
    \end{aligned}
\end{equation}
    has a unique global solution $u \in C([0,+\infty);\mathcal{H}^{1,1}_{rad})$ such that
    \begin{equation}\label{14}
        |u(t,x)|  \leq C(R) \varepsilon^{-1/2} |x|^{-\varepsilon} t^{-(1-\varepsilon)/2}
    \end{equation}
    for any $x \in \mathbb{R}^2 \setminus \{0\},$ any $t>1$ and for any $\varepsilon \in (0,1/2].$
\end{thm}
The  idea of the proofs of  these $L^\infty$ bounds
is based on the  pseudo - conformal estimate
    \begin{equation}
    \begin{aligned}
           & t^{-1} \|J(t) u(t)\|_{L^2}^2 \lesssim 1
        \end{aligned}
    \end{equation}
    where $t \geq 1$ and
    \begin{equation}
        J(t) = x+it\nabla = U(t)xU(-t) = M(t)it \nabla M(-t).
    \end{equation}

In addition, we turn to the case where no $L^\infty$ control is guaranteed.
In the case $n=2$ we have the following result that is a generalization of a result of Hirata in \cite{H95}. We use the standard notation of homogeneous Besov spaces $\dot{B}^{s}_{p,q}$ (see \cite{BL76} for definition and properties).
\begin{thm}\label{th5} Let $n=2.$
    For any $u_0 \in \mathcal{F}(H^1)$ there exists a unique $u_+ \in L^2,$ so  that
    \begin{equation}\label{wp1}
        \begin{aligned}
            \lim_{t \to +\infty}  \left\| U(-t) \exp \left(i \Phi_t(u) \right) u(t) - u_+\right\|_{\mathcal{F}(\dot{B}^{-1}_{2,\infty})} = 0,
            \end{aligned}
    \end{equation}
    where $\Phi_t$ is defined in \eqref{eq.phd}.
\end{thm}

Turning to the case $n \geq 3, $ we consider the problem
\begin{equation}\label{eq.NLSf4}
    \begin{aligned}
    & (i \partial_t + \frac{1}{2}\Delta_x)u = u ((1+|u|^{2})^{1/n}-1), \ \ t \geq 0; \\
    & u(0)= u_0.
    \end{aligned}
\end{equation}
The nonlinearity is associated with  the functions.
\begin{equation}\label{fvW}
    \begin{aligned}
        &f(z) = (\langle z\rangle^{2/n}-1) z, \\
        & V(z) = \frac{n}{n+1} (\langle z\rangle^{2(n+1)/n}-1)   - |z|^2, \\
        & W(z) = (n+2) V(z) -n \overline{z} f(z)
    \end{aligned}
\end{equation}
so that $f(z) = \partial_{\overline{z}} V$
and $f(0)=V(0)=W(0)=0.$
\begin{thm}\label{th6} Let $n \geq 3.$
    For any $u_0 \in \mathcal{F}(H^1)$ there exists a unique $u_+ \in L^2,$ so  that
    \begin{equation}\label{wpn2}
        \begin{aligned}
            \lim_{t \to +\infty}  \left\| U(-t) \exp \left(i \Phi_t(u) \right) u(t) - u_+\right\|_{\mathcal{F}(\dot{B}^{-n/2}_{2,\infty})} = 0,
            \end{aligned}
    \end{equation}
    where $u(t)$ is the solution to \eqref{eq.NLSf4}.
\end{thm}

\section{Notations}
The Fourier transform is defined by
\begin{equation} \label{eq.ftr1}
    \begin{aligned}
        \mathcal{F} \varphi(\xi) = (2\pi)^{-n/2} \int_{\mathbb{R}^n}   e^{-i x\cdot \xi} \varphi(x) dx
    \end{aligned}
\end{equation}
for $\varphi \in \mathcal{S},$ rapidly decreasing $C^\infty$ functions. Here and below $x \cdot \xi$ means the scalar product of $x,\xi \in \mathbb{R}^n$ in $\mathbb{R}^n.$ Often, we shall write $x^2$ in the place of $x\cdot x$ for $x \in \mathbb{R}^n.$

We use the notation $H^s$ for the classical Sobolev space
\begin{equation}\label{eq.clsob1}
    \begin{aligned}
        H^s =H^s(\mathbb{R}^n) = (1-\Delta)^{-s/2} L^2(\mathbb{R}^n)  \\
        = \mathcal{F}^{-1} (1+|\xi|^2)^{-s/2} \mathcal{F} L^2(\mathbb{R}^n).
    \end{aligned}
\end{equation}

For any $z = (z_1, \cdots,z_k) \in \mathbb{C}^k$ we use the notation
\begin{equation}
    \langle z \rangle = ( 1+|z|^2)^{1/2}, |z|^2 = \sum_{j=1}^k |z_j|^2.
\end{equation}

In the Hilbert space $L^2(\mathbb{R}^n;\mathbb{C}),$
 which consists of square-integrable complex-valued functions on $\mathbb{R}^n,$ the inner product is defined as:
\begin{equation}\label{ip1}
    \left. ( f \right \vert g ) = \int_{\mathbb{R}^n} f(x) \overline{g(x)} dx.
\end{equation}

The inequality $X \lesssim Y$ is used as a shorthand notation for inequalities that involve implicit constants. Specifically,
$X \lesssim Y$ means $X \leq C Y$ for some constant $C.$

\section{Preliminary calculations}
\label{s.3}

We use the Dollard decomposition \eqref{eq.doll5aa}, involving the operators $M(t),D(t)$ defined in \eqref{eq.MD1}.

We have the relations
\begin{equation}\label{eq.doll5}
\begin{aligned}
    & U(-t) = M(t)^{-1} \mathcal{F}^{-1}D(t)^{-1} M(t)^{-1},\\
    & \mathcal{F} M(t) U(-t) = D(t)^{-1} M(t)^{-1}.
\end{aligned}
\end{equation}

Note that multiplication by a function commutes with $M(t)$
\begin{equation}
  M(t)(gf) =  g   M(t)(f) .
\end{equation}
Moreover,
\begin{equation}\label{eq.dte5}
\begin{aligned}
  & (D(t)^{-1} \ e^{i\Phi_t} f)(x) = (it)^{n/2} e^{i\Phi_t(xt)} f \left(xt  \right)\\
  &= e^{i\tilde{\Phi}_t(x)} (D(t)^{-1} f)(x),
\end{aligned}
\end{equation}
where
\begin{equation}\label{tild1}
    \tilde{\Phi}_t(x) = \int_1^t  |u(\tau, \tau x)|^{2/n} d \tau.
\end{equation}

Then we set
\begin{equation}\label{eq.defvt5}
    v(t) = M(t) U(-t) e^{i \Phi_t(u)} u(t) .
\end{equation}

It is easy to deduce the following.
\begin{lem}\label{mle60}
    For any $g \in \mathcal{S}$ we have
   \begin{equation}
       \mathcal{F} M(t) U(-t) (g u(t)) = \tilde{g}   \mathcal{F} M(t) U(-t)  u(t) , \ \ \ \tilde{g}(x) = g(x t).
   \end{equation}
\end{lem}
\begin{proof}
    It is sufficient to apply \eqref{eq.doll5} that implies
    \begin{equation}\label{eq.dol68}
\begin{aligned}
    &\mathcal{F} M(t)U(-t) = D(t)^{-1} M(t)^{-1}.
\end{aligned}
\end{equation}
Then we can quote the following variant of
\eqref{eq.dte5}.
\begin{equation}\label{eq.dte75}
    D(t)^{-1} (gf)(x) = \tilde{g} D(t)^{-1}(f)(x)
\end{equation}
and complete the proof.
\end{proof}

We can apply the above Lemma and from \eqref{eq.defvt5} we can write
\begin{equation}
    v(t) = \mathcal{F}^{-1} \mathcal{F} M(t) U(-t) e^{i \Phi_t(u)} u(t)  =  \mathcal{F}^{-1} e^{i \tilde{\Phi}_t} \mathcal{F} M(t) U(-t)  u(t)
\end{equation}
due to \eqref{eq.doll5}.

In this way, we get
\begin{equation}\label{eq.meme0}
    \begin{aligned}
      &  i \partial_t v(t) = \mathcal{F}^{-1} i \partial_t \left(  e^{i \tilde{\Phi}_t} \mathcal{F} M(t) U(-t)  u(t)\right)  \\
        & = -  \mathcal{F}^{-1} \partial_t \tilde{\Phi}_t   e^{i \tilde{\Phi}_t}\mathcal{F} M(t) U(-t) u(t) \\
        & + i \mathcal{F}^{-1}   e^{i \tilde{\Phi}_t} \mathcal{F} \partial_t \left(M(t) U(-t) u(t)\right) .
    \end{aligned}
\end{equation}

On the other hand,
we have the integral representation
\begin{equation}
    u(t) = U(t) u_0 - i \int_0^t  U(t-\tau)  \ \ u(\tau)|u(\tau)|^{2/n}   \ \ d\tau .
\end{equation}
Hence
\begin{equation}
    i \partial_t \left( U(-t) u(t)\right) = U(-t) u(t)|u(t)|^{2/n}.
\end{equation}

Further we have the simple relation
\begin{equation}
    \partial_t (M(t) w(t)) = M(t) \left( \partial_t w(t) - i \frac{x^2}{2t^2} w(t)\right)
\end{equation}
hence
\begin{equation}
\begin{aligned}
   & i\partial_t \left( M(t) U(-t) u(t) \right) = M(t) U(-t) \left( u(t)|u(t)|^{2/n} \right) + \frac{x^2}{2t^2} M(t)U(-t)u(t).
\end{aligned}
\end{equation}
Now we turn back to \eqref{eq.meme0} and we can write
\begin{equation}\label{eq.meme1}
    \begin{aligned}
      &  i \partial_t v(t) =  -  \mathcal{F}^{-1}   e^{i \tilde{\Phi}_t} |u(t,xt)|^{2/n}\mathcal{F} M(t) U(-t) u(t) \\
        & +  \mathcal{F}^{-1}   e^{i \tilde{\Phi}_t} \mathcal{F} M(t) U(-t) \left( u(t)|u(t,x)|^{2/n} \right)  \\
        &+  \mathcal{F}^{-1}   e^{i \tilde{\Phi}_t} \mathcal{F} \frac{x^2}{2t^2} M(t)U(-t)u(t) \\
        &= -  \mathcal{F}^{-1}   e^{i \tilde{\Phi}_t} \mathcal{F} M(t) U(-t) \left(|u(t,xt)|^{2/n} u(t)\right) \\
        & +  \mathcal{F}^{-1}   e^{i \tilde{\Phi}_t} \mathcal{F} M(t) U(-t) \left( u(t)|u(t,x)|^{2/n} \right)  \\
        &+  \mathcal{F}^{-1}   e^{i \tilde{\Phi}_t} \mathcal{F} \frac{x^2}{2t^2} M(t)U(-t)u(t)  \\
        &= - \frac{1}{2t^2} \mathcal{F}^{-1}   e^{i \tilde{\Phi}_t} \Delta \mathcal{F}  M(t)U(-t)u(t).
    \end{aligned}
\end{equation}

Thus we obtain

\begin{cor}\label{cor3.1}
  \begin{equation}
    \begin{aligned}
    & i \frac{d}{dt} v(t)  = i \frac{d}{dt} \left( M(t)U(-t)  \exp \left(i \int_0^t |u(\tau, \tau x/t)|^{p-1} d\tau \right)   u(t)\right)  \\
     & = \mathcal{F}^{-1}  \exp \left(i \int_0^t |u(\tau, \tau x)|^{p-1} d\tau \right) \frac{(-\Delta)}{2t^2}D(t)^{-1} M(-t) u(t)  \\
     & = \mathcal{F}^{-1}  \exp \left(i \int_0^t |u(\tau, \tau x)|^{p-1} d\tau \right) \frac{(-\Delta)}{2t^2}\mathcal{F} M(t) U(-t) u(t).
    \end{aligned}
\end{equation}
\end{cor}

\section{$L^\infty$ bound implies the existence of modified wave operator}

\begin{proof}[Proof of Theorem \ref{MT1} ]

We start with the plan of the proof. It consists of the following steps.

\begin{enumerate}
    \item[Step I)] We have to  prove that
\begin{equation}\label{eq.defvt47}
    v(t) = M(t) U(-t) e^{i \Phi_t(u)} u(t)
\end{equation}
has $L^2$ limit as $t \to +\infty.$  Indeed, the $L^2$ limit in \eqref{mwoe60} is valid if
$$ U(-t) e^{i \Phi_t(u)} u(t) $$
is $L^2$ Cauchy sequence as $t \to +\infty.$ Taking into account the fact that
\begin{equation}
    \lim_{t \to +\infty}\|M(t) g - g\|_{L^2} = 0, \ \ \forall g \in L^2
\end{equation}
due to Lebesgue convergence theorem, we see that we have to check that $v(t)$ is a Cauchy sequence.
Therefore in this step we shall verify this property assuming $n \leq 2$ we have the $L^\infty$ control \eqref{linf5} as well as we have the uniform estimates
 \begin{equation}\label{eq.jco6}
        \|J(t+1) u(t)\|_{L^2} \lesssim  (1+t)^{1/2} \|u(0)\|_{\mathcal{H}^{0, 1}}  .
    \end{equation}
    \begin{equation}\label{eq.jco29}
     \int_0^t  (1+\tau)^{-2} \|J(1+\tau) u(\tau)\|_{L^2}^2 d\tau \lesssim  \|u(0)\|^2_{\mathcal{H}^{0,1}}.
    \end{equation}
Here and below $J(1+t)$ is the generator of the pseudoconformal group
    \begin{equation} \label{eq.nmn}
    J(1+t) = x+i(1+t) \nabla_x ;
\end{equation}

    \item[Step II)] We prove \eqref{eq.jco6}  and \eqref{eq.jco29} by using pseudoconformal transform and applying simple energy estimate.
\end{enumerate}

Now we turn to the  details of  the proof of Theorem \ref{MT1}.

     Note that $v(t)$ satisfies
the relations (see Lemma \ref{mle60})
\begin{equation}
\begin{aligned}
   \mathcal{F}(v(t)) =    \mathcal{F}  M(t) U(-t) e^{i \Phi_t(u)} u(t) \\
   = e^{i \tilde{\Phi}_t} \mathcal{F}  M(t) U(-t)  u(t)= e^{i \tilde{\Phi}_t}
\mathcal{F}  M(t) U(-t)  u(t) .
\end{aligned}
\end{equation}
The conservation of mass guarantees that it is sufficient to prove that
\begin{equation}
 \frac{ \|v(t)-v(s)\|_{L^2}^2}{2} =  \|v(s)\|_{L^2}^2 - \left. \mathrm{Re} ( v(t) \right \vert| v(s) )
\end{equation}
is small if $t>s>M$ and  if $M$ is large.

We use Corollary \ref{cor3.1} and setting
\begin{equation}\label{eq.bw1}
    \mathbb{W}(t) =  M(t) U(-t) u(t),
\end{equation}
we can write
\begin{equation}
    \begin{aligned}
       & \|v(s)\|_{L^2}^2 - \mathrm{Re} \left( v(t) \big| v(s) \right) =
        -  \mathrm{Re} \left( v(t)-v(s) | v(s) \right)  \\
       &= -\mathrm{Re}\int_s^t \left( v^\prime(\tau)| v(s) \right) d\tau  \\
       & = \left. \mathrm{Im} \int_s^t  \left(  \mathcal{F}^{-1}  \exp \left(i \tilde{\Phi}_\tau  \right) \frac{(-\Delta)}{2\tau^2}\mathcal{F} \mathbb{W}(\tau) \right \vert  M(s) U(-s) e^{i \Phi_s} u(s)  \right) d\tau\\
       & = \left.\mathrm{Im}  \int_s^t  \left(    \exp \left(i \tilde{\Phi}_\tau  \right) \frac{(-\Delta)}{2\tau^2}\mathcal{F} \mathbb{W}(\tau) \right \vert e^{i \tilde{\Phi}_s} \mathcal{F} \mathbb{W}(s)  \right) d\tau  \\
       & = \left. \mathrm{Im}  \int_s^t  \left(    \exp \left(i \tilde{\Phi}_\tau - i\tilde{\Phi}_s  \right) \frac{(-\Delta)}{2\tau^2}\mathcal{F} \mathbb{W}(\tau) \right \vert  \mathcal{F} \mathbb{W}(s) \right) d\tau  \\
       & =\left.  \mathrm{Im}  \int_s^t  \left(    \exp \left(i \tilde{\Phi}_\tau - i\tilde{\Phi}_s  \right) \frac{\nabla_x}{2\tau^2}\mathcal{F}\mathbb{W}(\tau)  \right \vert \nabla_x \mathcal{F}\mathbb{W}(s)  \right) d\tau  \\
       &+ \left. \mathrm{Im}  \int_s^t  \left(   \nabla_x \left[ \exp \left(i \tilde{\Phi}_\tau - i\tilde{\Phi}_s  \right) \right] \frac{\nabla_x}{2\tau^2}\mathcal{F} \mathbb{W}(\tau) \right \vert  \mathcal{F} \mathbb{W}(s)  \right) d\tau .
    \end{aligned}
\end{equation}
Setting
\begin{equation}
    \begin{aligned}
      &  I(t,s) = \left.  \mathrm{Im}  \int_s^t  \left(    \exp \left(i \tilde{\Phi}_\tau - i\tilde{\Phi}_s  \right) \frac{\nabla_x}{2\tau^2}\mathcal{F} \mathbb{W}(\tau) \right \vert \nabla_x \mathcal{F}\mathbb{W}(s)  \right) d\tau , \\
        & II(t,s) = \left. \mathrm{Im}  \int_s^t  \left(   \nabla_x \left[ \exp \left(i \tilde{\Phi}_\tau - i\tilde{\Phi}_s  \right) \right] \frac{\nabla_x}{2\tau^2}\mathcal{F}\mathbb{W}(\tau) \right \vert  \mathcal{F} \mathbb{W}(s)  \right) d\tau,
    \end{aligned}
\end{equation}
where here and below we use the definition of $\mathbb{W}$ from \eqref{eq.bw1}.
We can estimate the first term as follows
\begin{equation}
    \begin{aligned}
        |I(t,s)| \lesssim  \int_s^t \left \| \frac{(\nabla_x)}{2\tau^2}\mathcal{F} \mathbb{W}(\tau) \right\|_{L^2}\left\| \nabla_x \mathcal{F}\mathbb{W}(s) \right \|_{L^2}d\tau .
    \end{aligned}
\end{equation}
Note that
\begin{equation}
    \begin{aligned}
        & \left \| \nabla_x \mathcal{F} \mathbb{W}(\tau) \right\|_{L^2} = \left \| \nabla_x \mathcal{F} M(\tau) U(-\tau) u(\tau) \right\|_{L^2}\\
        &=  \left \| x M(\tau) U(-\tau) u(\tau) \right\|_{L^2}=  \left \|U(\tau) x U(-\tau) u(\tau) \right\|_{L^2} = \left \|J(\tau)  u(\tau) \right\|_{L^2},
    \end{aligned}
\end{equation}
where
\begin{equation}
    J(t) = U(t) x U(-t)  = x+it \nabla_x
\end{equation}
and
\begin{equation}\label{eq.tau12}
\begin{aligned}
    & \left \| J(\tau) u(\tau) \right\|_{L^2} \leq  \left \| J(1+\tau) u(\tau) \right\|_{L^2} + \left \| \nabla u(\tau) \right\|_{L^2}\\
    & \lesssim \tau^{1/2} + O(1),  \ \ \tau \to \infty
\end{aligned}
\end{equation}
due to uniform bounds of Proposition \ref{lem.mapes}.
 Then we find
\begin{equation}
    \begin{aligned}
      &  |I(t,s)| \lesssim  \int_s^t \tau^{-2}  \left \| J(\tau) u(\tau) \right\|_{L^2}  \left\|J  u(s) \right \|_{L^2}d\tau \\
      &\lesssim s^{1/2} \int_s^t \tau^{-2} \left \| J(\tau)  u(\tau) \right\|_{L^2} d\tau  \\
      & \lesssim s^{1/2} \left(\int_s^t \frac{d\tau}{\tau^2}\right)^{1/2} \left(\int_s^t \left \| J(\tau) u(\tau) \right\|_{L^2}^2 \frac{d\tau}{\tau^2}\right)^{1/2} = o(1)
    \end{aligned}
\end{equation}
as $s \to +\infty.$
To estimate the second term we proceed as follows
\begin{equation}
    \begin{aligned}
        | II(t,s) | \lesssim \int_s^t   \left\| \nabla_x \left(i \tilde{\Phi}_\tau - i\tilde{\Phi}_s  \right) \right\|_{L^2}
   \left\|\frac{(\nabla_x)}{2\tau^2}\mathcal{F} \mathbb{W}(\tau)\right\|_{L^2}  \|\mathcal{F} \mathbb{W}(s)\|_{L^\infty}  d\tau.
    \end{aligned}
\end{equation}
On the other hand, we use \eqref{eq.doll5} and deduce
\begin{equation}
\begin{aligned}
    &\|\mathcal{F} \mathbb{W}(s)\|_{L^\infty} =\|\mathcal{F} M(s) U(-s)  u(s)\|_{L^\infty}  = \|D(-s)M(-s)  u(s)\|_{L^\infty}    \\
    &=\|D(-s)  u(s)\|_{L^\infty} = s^{n/2} \| u(s)\|_{L^\infty} = O(1).
\end{aligned}
\end{equation}
Further, we can write
\begin{equation}\label{eq.nax3}
\begin{aligned}
    &\nabla_x \left( u(\sigma, \sigma x) e^{i x^2 \sigma/2} \right)       \\
    & = e^{i x^2 \sigma/2}  \left[  \sigma (\nabla_y u)(\sigma,\sigma x) +i\sigma x u(\sigma,\sigma x) \right]\\
    & = i e^{i x^2 \sigma/2}  (J(\sigma) u(\sigma))( \sigma x),
\end{aligned}
\end{equation}
so
we get
\begin{equation}
\begin{aligned}
   &  \left| \nabla_x \left(i \tilde{\Phi}_\tau - i\tilde{\Phi}_s  \right) \right| \lesssim
   \int_s^{\tau} \left|\nabla |u(\sigma, \sigma x)|^{2/n} \right|  d\sigma \\
   & = \int_s^{\tau} \left|\nabla_x |e^{-ix^2\sigma/4} u(\sigma, \sigma x)|^{2/n} \right|  d\sigma \\
   & \lesssim \int_s^{\tau} \left|\nabla_x |e^{-ix^2\sigma/4} u(\sigma, \sigma x)|^{2/n} \right|  d\sigma \\
   & \lesssim \int_s^{\tau}  |(J(\sigma) u(\sigma)) ( \sigma x)| |u(\sigma,\sigma x)|^{2/n-1}   d\sigma  \\
   & \lesssim \int_s^{\tau}  |(J(\sigma) u(\sigma)) ( \sigma x)| \sigma^{-1+n/2}  d\sigma.
\end{aligned}
\end{equation}
Taking the $L^2$ norm we find
\begin{equation}\label{eq.jnam9}
\begin{aligned}
   &  \left\| \nabla_x \left(i \tilde{\Phi}_\tau - i\tilde{\Phi}_s  \right) \right  \|_{L^2}  \\
   & \lesssim \int_s^{\tau}  \|(J(\sigma) u(\sigma)) (\sigma \cdot)\|_{L^2} \ \sigma^{-1+n/2}  d\sigma \\
   & = \int_s^{\tau}  \|J(\sigma) u(\sigma)\|_{L^2}   \frac{d\sigma}{\sigma}.
\end{aligned}
\end{equation}
Therefore,  using \eqref{eq.tau12}, we arrive at
\begin{equation}\label{intes7}
    \begin{aligned}
      &  | II(t,s) | \lesssim \int_s^t \tau^{-2} \int_s^{\tau}  \|J(\sigma) u(\sigma)\|_{L^2}   \frac{d\sigma}{\sigma}
   \left\|J(\tau) u(\tau)\right\|_{L^2}    d\tau \\
   & \lesssim \left(\int_s^t \tau^{-2} \left(\int_s^{\tau}  \|J(\sigma) u(\sigma)\|_{L^2}   \frac{d\sigma}{\sigma}\right)^2 d \tau \right)^{1/2} \left( \int_s^t \tau^{-2} \left\|J(\tau) u(\tau)\right\|_{L^2}^2 d\tau\right)^{1/2}.
    \end{aligned}
\end{equation}
Now recall a minor modification of the classical Hardy inequality.
\begin{lem}
    If $F(s) \in C(0,+\infty) \cap L^2(0,+\infty),$ then
    \begin{equation}\label{H1}
     \begin{aligned}
      \int_s^t \tau^{-2} \left(\int_s^{\tau}  F(\sigma)  d\sigma \right)^2 d \tau \leq 4 \int_s^t    F(\tau)^2 d \tau
    \end{aligned}
\end{equation}
for $0<s<t<+\infty.$
\end{lem}
\begin{proof}
    Set $f(\tau) = F(s+\tau).$ Then the inequality \eqref{H1} is transformed into
    \begin{equation}\label{H4}
     \begin{aligned}
      \int_0^{t-s} (s+\tau)^{-2} \left(\int_0^{\tau}  f(\sigma)  d\sigma \right)^2 d \tau \leq 4 \int_0^{t-s}    f(\tau)^2 d \tau
    \end{aligned}
\end{equation}
Since
\begin{equation}
    (s+\tau)^{-1} < \tau^{-1},
\end{equation}
we see that  the classical Hardy inequality (see Chapter IX, section 327, inequality (9.8.2) in \cite{H52})
\begin{equation}\label{H8}
     \begin{aligned}
      \int_0^{+\infty} \tau^{-2} \left(\int_0^{\tau}  f(\sigma)  d\sigma \right)^2 d \tau \leq 4 \int_0^{+\infty}    f(\tau)^2 d \tau
    \end{aligned}
\end{equation}
implies \eqref{H4}. This completes the proof.
\end{proof}

We apply \eqref{H1} with $F(\sigma) = \sigma^{-1}\|J(\sigma) u(\sigma)\|_{L^2}  $ and via \eqref{eq.tau12} we arrive at
\begin{equation}\label{intes51}
    \begin{aligned}
      &  | II(t,s) | \lesssim \int_s^t \tau^{-2} \left\|J(\tau) u(\tau)\right\|_{L^2}^2 d\tau  = o(1).
    \end{aligned}
\end{equation}

\end{proof}

\section{Proof of \eqref{eq.jco6}  and \eqref{eq.jco29} via pseudoconformal transformation}

\subsection{Main Proposition.}

The key point in this section is the following.
\begin{prop}\label{lem.mapes}
   The solution $u$ to \eqref{eq.NLSf1}
satisfies
\begin{enumerate}
    \item the conservation of the mass
\begin{equation}
    \|u(t)\|_{L^2}^2 = \|u_0\|_{L^2}^2,
\end{equation}
\item the  conservation of the pseudoconformal energy
    \begin{equation}\label{CPCE1}
    \begin{aligned}
           & (1+t)^{-1} \|J(1+t) u(t)\|_{L^2}^2 + \frac{2n}{n+1} (1+t) \|u(t)\|_{L^{2+2/n}}^{2+2/n}\\
           &  + \int_0^t (1+s)^{-2} \|Ju(s)\|_{L^2}^2 ds \\
            &=  \|J(1) u(0)\|_{L^2}^2 + \frac{2n}{n+1}  \|u(0)\|_{L^{2+2/n}}^{2+2/n},
        \end{aligned}
    \end{equation}
    where
\begin{equation}
        J(t+1) = x+i(1+t)\nabla.
    \end{equation}
\end{enumerate}
\end{prop}

\begin{remark}
    Since we use pseudoconformal transformation, we follow \cite{FM17} and the method of the proof in this work is based on the approach in \cite{O06}.
\end{remark}

\subsection{Pseudoconformal transformation}
The first known application of the pseudoconformal transformation to the NLS dates back to the 1980s,  (see \cite{GV1980}, \cite{TY84},  \cite{W85}, \cite{CW92}).
We
use  the  following pseudoconformal transformation $(t,x,u) \mapsto (\tau,\xi,v)$
defined by
\begin{equation}\label{eq.pct89f}
\begin{aligned}
   & t=\frac{\tau}{1-\tau},\ x = \frac{\xi}{1-\tau},\\
    & u(t,x) =(1+t)^{-n/2} v\left(\frac{t}{1+t},\frac{x}{1+t} \right)  e^{\frac{ix^2}{2(1+t)}}
    \end{aligned}
\end{equation}
with inverse transformation given by
\begin{equation}\label{eq.pct90f}
\begin{aligned}
   & \tau = \frac{t}{1+t},\ \xi = \frac{x}{1+t},\\
    &  v(\tau,\xi) =    u\left(\frac{\tau}{1-\tau},\frac{\xi}{1-\tau}\right) (1-\tau)^{-n/2} e^{-\frac{i\xi^2}{2(1-\tau)}}.
    \end{aligned}
\end{equation}

If $u$ is a solution to \eqref{eq.NLSf1}, then
we can apply the pseudoconformal transformation and we can obtain
\begin{equation}\label{eq.meq50}
    \begin{aligned}
    & (i \partial_\tau + \frac{1}{2}\Delta_\xi)v = \frac{v(\tau)|v(\tau)|^{2/n}}{1-\tau}, \ \ \tau \in [0,1),\\
    & v(0)= v_0 ,
    \end{aligned}
\end{equation}
where
\begin{equation}\label{eqv0}
    v_0(x) = u_0(x) e^{-ix^2/2}.
\end{equation}
\begin{lem}
   We have the relations
\begin{equation}\label{eq.mrel21}
\begin{aligned}
  &\|u(t)\|_{L^2} = \|v(\tau)\|_{L^2}, \ \\
  & \left\|u\left(t\right)\right\|^{2(n+1)/n}_{L^{2(n+1)/n}} =    (1+t)^{-1} \left\|v\left(\tau \right)\right\|^{2(n+1)/n}_{L^{2(n+1)/n}}, \\
  &  \| \xi v(\tau)\|_{ L^2} =  (1+t)^{-1} \| x u(t)\|_{ L^2},  \\
  & \| \nabla_{\xi} v(\tau)\|_{ L^2} = \|J(t+1)u(t)\|_{ L^2},
\end{aligned}
\end{equation}
where
\begin{equation}
   \tau =  \frac{t}{1+t}, \ \     J(t+1) = x+i(1+t)\nabla.
    \end{equation}
\end{lem}

\begin{proof}
From \eqref{eq.pct89f}  we have

\begin{equation}\label{eq.PCT1}
    \ |u(t,x)| =  \left|v\left(\frac{t}{1+t},\frac{x}{1+t}\right)\right| (1+t)^{-n/2}= \left|v\left(\tau,\frac{x}{1+t}\right)\right| (1+t)^{-n/2},
\end{equation}
and
we find
\begin{equation}
    \begin{aligned}
         & \int_{\mathbb{R}^n} |u(t,x)|^2 dx = \int_{\mathbb{R}^n} (1+t)^{-n}\left|v\left(\tau,\frac{x}{1+t}\right)\right|^2 dx =\left\|v\left(\tau \right)\right\|^2_{L^2}
    \end{aligned}
\end{equation}
and
\begin{equation}
    \begin{aligned}
         &\int_{\mathbb{R}^n} |u(t,x)|^q dx =  \int_{\mathbb{R}^n}(1+ t)^{-nq/2}\left|v\left(\tau,\frac{x}{1+t}\right)\right|^q dx\\
         &= (1+t)^{n-nq/2} \left\|v\left(\tau \right)\right\|^q_{L^q}.
    \end{aligned}
\end{equation}
Taking  $q= 2+2/n =2(n+1)/n ,$ we find
\begin{equation}
\left\|u\left(t\right)\right\|^{2(n+1)/n}_{L^{2(n+1)/n}} =    (1+t)^{-1} \left\|v\left(\tau \right)\right\|^{2(n+1)/n}_{L^{2(n+1)/n}}.
\end{equation}
In a similar way, we have
\begin{equation}
    \begin{aligned}
         & \int_{\mathbb{R}^n} x^2|u(t,x)|^2 dx = \int_{\mathbb{R}^n} (1+t)^{-n+2}\left|\frac{x}{1+t} v\left(\tau,\frac{x}{1+t}\right)\right|^2 dx \\
         & =(1+t)^2\left\||\xi|v\left(\tau\right)\right\|^2_{L^2}.
    \end{aligned}
\end{equation}
Finally, we have
\begin{equation}\label{eq.mrel22}
\begin{aligned}
   &   \nabla_x u(t,x) = \\
   & =  e^{i\xi^2/2(1-\tau)} (1-\tau)^{n/2} \left(i\xi+ (1-\tau)\nabla_\xi \right)v\left( \tau, \xi \right),\\
   & (1-\tau)^{n/2}\nabla_\xi v(\tau,\xi) \\
   & = e^{-ix^2/(2(1+t))}   \left(-ix +(1+t)\nabla_x  \right)u(t,x) .    \\
\end{aligned}
\end{equation}

Hence,
\begin{equation}
    \begin{aligned}
  & \int_{\mathbb{R}^n} |\nabla_\xi v(\tau,\xi)|^2 d\xi\\
  & = \int_{\mathbb{R}^n} |\left(x +i (1+t)\nabla_x  \right)u(t,x)|^2 dx = \|J(1+t)u(t)\|_{L^2}^2. \\
    \end{aligned}
\end{equation}

\end{proof}

\begin{remark}
   The above relations \eqref{eq.mrel21} and \eqref{eq.mrel22} show that
\begin{equation}\label{eq.ind12}
   u_0 \in \mathcal{H}^{1,1}   \ \ \mbox{if and only if } \ \  v_0(x) = u_0(x) e^{-ix^2/2} \in \mathcal{H}^{1,1}.
\end{equation}
\end{remark}

After the pseudoconformal transformation we find the Cauchy problem \eqref{eq.meq50} for $v.$
The precise derivation of energy conservation can be done by following \cite{O06}, \cite{FM17}.

\subsection{Conservation of the pseudoconformal energy}

In alternative way we can approximate the initial data with data in $H^2$ we can assume that
$$ v \in C([0,1);H^2(\mathbb{R}^n)) \cap C^1([0,1);L^2(\mathbb{R}^n)).$$
Multiplying (in $L^2$) the equation
\begin{equation}\label{eq.meq93}
  (1-\tau)  (i \partial_\tau +\frac{1}{2}\Delta_\xi)v = v|v|^{2/n}
\end{equation}
by $\partial_t \overline{v}$ and taking the real part, we find
\begin{equation}\label{eq.meq7}
    \frac{(1-\tau)}{4} \frac{d}{d\tau} \|\nabla_\xi v(\tau)\|_{L^2}^2 + \frac{d}{d\tau}\frac{n}{2(n+1)} \|v(\tau)\|_{L^{2(n+1)/n}}^{2(n+1)/n} = 0
\end{equation}
or
\begin{equation}\label{eq.tau1}
\begin{aligned}
   & \frac{d}{d\tau} \left(  \frac{(1-\tau)}{4}  \|\nabla v(\tau)\|_{L^2}^2 + \frac{n}{2(n+1)} \|v(\tau)\|_{L^{2(n+1)/n}}^{2(n+1)/n}\right)\\
    &+\frac{1}{4} \|\nabla u(\tau)\|_{L^2}^2  = 0.
\end{aligned}
\end{equation}
Passing to $t=\tau/(1-\tau),$ we find
$\partial_\tau = (1+t)^2 \partial_t$, so
\begin{equation}\label{eq.meq7mmm}
\begin{aligned}
    &\frac{(1+t)^2}{4} \frac{d}{dt} \left(\frac{1}{1+t}\|J(1+t)u(t)\|_{L^2}^2 \right) \\
    &+ (1+t)^2 \frac{d}{dt}\left(\frac{n(1+t)}{2(n+1)} \|u(t)\|_{L^{2(n+1)/n}}^{2(n+1)/n}\right)\\
    &+ \frac{1}{4}\|J(1+t)u(t)\|_{L^2}^2  = 0
\end{aligned}
\end{equation}
and
\begin{equation}\label{eq.tau2}
\begin{aligned}
   &\frac{d}{dt} \left(  \frac{1}{4(1+t)}  \|J(1+t)u(t)\|_{L^2}^2 + \frac{n(1+t)}{2(n+1)} \|u(t)\|_{L^{2(n+1)/n}}^{2(n+1)/n}\right)\\
    &+\frac{1}{4(1+t)^2} \|J(1+t)u(t)\|_{L^2}^2  = 0.
\end{aligned}
\end{equation}

Integrating \eqref{eq.tau1} from $0$ to $\tau$, we find
\begin{equation}
\begin{aligned}
  &  (1-\tau)  \|\nabla v(\tau)\|_{L^2}^2+\frac{2n}{(n+1)} \|v(\tau)\|_{L^{2(n+1)/n}}^{2(n+1)/n} + \int_0^\tau \|\nabla v(\theta)\|_{L^2}^2 d \theta =\\
  &   \|\nabla v(0)\|_{L^2}^2+ \frac{2n}{(n+1)} \|v(0)\|_{L^{2(n+1)/n}}^{2(n+1)/n}.
\end{aligned}
\end{equation}
Integrating \eqref{eq.tau2}, we obtain
\begin{equation}\label{eq.rr86}
\begin{aligned}
   & (1+t)^{-1}  \|J(1+t) u(t)\|^2_{L^2}+ \frac{2 n}{(n+1)} (1+t) \|u(t)\|_{L^{2(n+1)/n}}^{2(n+1)/n}\\
   & + \int_0^t\|J(1+\sigma)u(\sigma)\|^2_{L^2}(1+\sigma)^{-2} d\sigma\\
   & =\|J(1)u(0)\|^2_{L^2} +\frac{2 n}{(n+1)}  \|u(0)\|_{L^{2(n+1)/n}}^{2(n+1)/n}.
\end{aligned}
\end{equation}

Note that the proof gives the following energy estimate for $v.$
\begin{lem}\label{lem.L1.1}
Let  $v$ be a classical solution to \eqref{eq.meq50} with initial data in $H^1.$ Then for any $\tau \in [0,1)$ we have the identity
 \begin{equation}\label{eq.ee5}
\begin{aligned}
    & (1-\tau)  \|\nabla v(\tau)\|^2_{L^2}+ \frac{2n}{n+1}  \|v(\tau)\|^{2+2/n}_{L^{2+2/n}}\\
    & + \int_0^\tau \|\nabla v(\sigma)\|^2_{L^2} d\sigma  = C_* ,  \\
    & C_* =     \|\nabla v(0)\|^2_{L^2} + \frac{2n}{n+1}  \|v(0)\|^{2+2/n}_{L^{2+2/n}}.
\end{aligned}
\end{equation}
\end{lem}

\begin{proof}[Proof of Propostition \ref{lem.mapes}]

From \eqref{eq.rr86} we conclude that \eqref{CPCE1} is true.

\end{proof}

\section{Proof of Theorem \ref{MT2}}
We have the following relation.
\begin{equation}
    \begin{aligned}
        &\left( u^+ |\partial_j \varphi \right)= \left( \left. u^+- U(-t) e^{i \Phi_t(u)} u(t)    \right\vert \partial_j \varphi \right)\\
        & -
        \left(\partial_j \left(\left. U(-t) e^{i \Phi_t(u)} u(t)\right)    \right\vert  \varphi \right).
    \end{aligned}
\end{equation}
Hence, we have the estimates
\begin{equation}\label{eq.f1}
    \begin{aligned}
       & \left| \left( u^+ |\partial_j \varphi \right) \right| \leq \left\|  u^+- U(-t) e^{i \Phi_t(u)} u(t)  \right\|_{L^2} \|\partial_j \varphi\|_{L^2}\\
        &+
        \left(\|\partial_j \Phi_t\|_{L^2} \|u(t)\|_{L^\infty} + \|\partial_j u(t)\|_{L^2} \right) \|\varphi\|_{L^2}.
    \end{aligned}
\end{equation}
Next, we use the relations
\begin{equation}
    \begin{aligned}
        &\partial_j \left|u\left(\tau, \frac{\tau x}{t} \right)\right|^{2/n} = \frac{2\tau}{n t} \left|u\left(\tau, \frac{\tau x}{t} \right)\right|^{2/n-2}
        \mathrm{Re} \left(\overline{u\left(\tau, \frac{\tau x}{t} \right)} \ \partial_j u\left(\tau, \frac{\tau x}{t} \right)  \right)\\
      &  = -\frac{2}{n t} \left|u\left(\tau, \frac{\tau x}{t} \right)\right|^{2/n-2}
        \mathrm{Im } \left(\overline{u\left(\tau, \frac{\tau x}{t} \right)} \ J_j(\tau) u\left(\tau, \frac{\tau x}{t} \right)  \right),
    \end{aligned}
\end{equation}
where
\begin{equation} \label{eq.sc3}
    J_j(\tau)u(\tau,x) = (i\tau  \partial_{x_j} + x_j)u(\tau,x).
\end{equation}
 Hence, we have
\begin{equation}
    \begin{aligned}
     &   \|\partial_j \Phi_t \|_{L^2} \leq \frac{2}{n t} \int_1^t \|u(\tau)\|_{L^\infty}^{2/n-1} \|J_j  u\left(\tau, \frac{\tau \  \cdot}{t}\right)\|_{L^2} d \tau\\
        & \leq \frac{2}{n t} \int_1^t \|u(\tau)\|_{L^\infty}^{2/n-1} \left(\frac{\tau}{t}\right)^{-n/2}\|J_j(\tau)  u\left(\tau\right)\|_{L^2} d \tau\\
        & \leq Ct^{n/2-1} \int_1^t \tau^{-1}  \|J_j(\tau)  u\left(\tau\right)\|_{L^2} d \tau \lesssim t^{(n-1)/2},
    \end{aligned}
\end{equation}
where in the last step, we have used the  estimates of Proposition \ref{lem.mapes}.
From \eqref{eq.f1} we now obtain
\begin{equation}\label{eq.f2}
    \begin{aligned}
       & \left| \left( u^+ |\partial_j \varphi \right) \right| \leq \left\|  u^+- U(-t) e^{i \Phi_t(u)} u(t)  \right\|_{L^2} \|\partial_j \varphi\|_{L^2}\\
        &+
        \left(Ct^{-1/2} + 2E(u_0) \right) \|\varphi\|_{L^2}
    \end{aligned}
\end{equation}
so taking the limit $t \to +\infty$ we arrive at
\begin{equation}\label{eq.f3}
    \begin{aligned}
       & \left| \left( u^+ |\partial_j \varphi \right) \right|\lesssim
         2E(u_0)  \|\varphi\|_{L^2}
    \end{aligned}
\end{equation}
and we can conclude that $u^+$ is in $H^1.$

In order to show the required estimate \eqref{mwoe60m}, we approximate $\varphi \in H^1$ with a sequence in $H^2,$ so that
\begin{equation}
    \begin{aligned}
        \varphi_\varepsilon \in H^2, \ \ \lim_{\varepsilon \to 0} \|\varphi_\varepsilon - \varphi\|_{H^1} =0.
    \end{aligned}
\end{equation}
As above, we have the estimates
\begin{equation}\label{eq.f4}
    \begin{aligned}
       & \left| \left( \nabla \left( U(-t) e^{i \Phi_t(u)} u(t) - u^+\right)  |\nabla \varphi \right) \right| \\
       &\leq \left(\left\|  \nabla U(-t) e^{i \Phi_t(u)} u(t) \right\|_{L^2} + \|\nabla u^+\|_{L^2}\right) \|\nabla(\varphi_\varepsilon - \varphi)\|_{L^2} \\
       & + \left\| U(-t) e^{i \Phi_t(u)} u(t) - u^+ \right\|_{L^2} \|\Delta \varphi_\varepsilon \|_{L^2} \\
       & \leq C\left(t^{-1/2} + E(u_0)\right) \|\nabla(\varphi_\varepsilon - \varphi)\|_{L^2} \\
       &+ \left\| U(-t) e^{i \Phi_t(u)} u(t) - u^+ \right\|_{L^2} \|\Delta \varphi_\varepsilon \|_{L^2} .
    \end{aligned}
    \end{equation}
This leads to the estimate
\begin{equation}
    \begin{aligned}
        \limsup_{t \to +\infty} \left| \left( \nabla \left( U(-t) e^{i \Phi_t(u)} u(t) - u^+\right)  |\nabla \varphi \right) \right| \leq CE(u_0) \|\nabla (\varphi - \varphi_\varepsilon)\|_{L^2}
    \end{aligned}
\end{equation}
so we have \eqref{mwoe60m}.

\section{$L^\infty$ bounds}
We assume that $u$ is solution to the Cauchy problem \eqref{eq.NLSf1} with initial data in $\mathcal{H}^{1,1}.$

We have the following conservation laws.

\begin{itemize}
    \item Charge: $$ \|u(t)\|_{L^2}^2 = \|u_0\|_{L^2}^2.$$

    \item Energy : $$ \|\nabla u(t)\|_{L^2}^2 + \frac{2n}{n+1} \|u(t)\|_{L^{2+2/n}}^{2+2/n} = \|\nabla u_0\|_{L^2}^2 + \frac{2n}{n+1} \|u_0\|_{L^{2+2/n}}^{2+2/n}.$$

    \item Pseudo - conformal charge:
    \begin{equation}
    \begin{aligned}
           & (1+t)^{-1} \|J(1+t) u(t)\|_{L^2}^2 + \frac{2n}{n+1} (1+t) \|u(t)\|_{L^{2+2/n}}^{2+2/n}\\
           &+ \int_0^t (1+s)^{-2} \|J(s+1)u(s)\|_{L^2}^2 ds \\
            &=  \|J(1) u(0)\|_{L^2}^2 + \frac{2n}{n+1}  \|u(0)\|_{L^{2+2/n}}^{2+2/n},
        \end{aligned}
    \end{equation}
    where
    \begin{equation}
        J(t) = x+it\nabla = U(t)xU(-t) = M(t)it \nabla M(-t).
    \end{equation}
\end{itemize}

\begin{proof}[Proof of Theorem \ref{T.1}]
   Now we use the estimates \footnote{We use the following estimates. For any $\varphi \in H^1(\mathbb{R})$
   \begin{equation}
       \begin{aligned}
           &|\varphi(x)|^p = - \int_x^{+\infty}  \frac{d}{dy}  \left(|\varphi(y)|^p \right) dy \\
           &\leq p \int_{-\infty}^{+\infty} |\varphi|^{p-1} |\varphi^\prime| dy \leq p \|\varphi\|_{L^{2p-2}}^{p-1} \|\varphi^\prime\|_{L^2} . \ \
       \end{aligned}
 \end{equation}
 Hence
 \begin{equation}
     \|\varphi\|_{L^\infty} \leq p^{1/p}  \|\varphi\|_{L^{2p-2}}^{1-1/p} \|\varphi^\prime\|_{L^2}^{1/p}.
 \end{equation}

   }
\begin{equation}
    \begin{aligned}
        &\|u(t)\|_{L^\infty}^3 = \|M(-t)u(t)\|_{L^\infty}^3 \leq 3 \|M(-t)u(t)\|^2_{L^4} \|\nabla M(-t)u\|_{L^2} \\
        & = 3 \|u(t)\|^2_{L^4} \left(t^{-1} \|J(t)u(t)\|_{L^2}\right) \leq C (t^{-1/4})^2 (t^{-1} t^{1/2} ) = C t^{-1}
    \end{aligned}
\end{equation}
for $ t \geq 1.$
Hence we have \eqref{13}.
\end{proof}

\begin{proof}[Proof of Theorem \ref{T.2}]
    We take $u_0 \in \mathcal{H}^{1,1}$ radial. Then for any $x \in \mathbb{R}^2 \setminus \{0\}$, $p \geq 2,$ and $t \geq 1$ we have \footnote{
Here for any $\varphi(r) \in H^1(\mathbb{R}),$ any $r>0$ and any $p \geq 2$ we use the inequalities
\begin{equation}
    \begin{aligned}
        &|\varphi(r)|^p \leq p \left( \int_r^{+\infty} |\varphi(p)|^{2p-2} dp\right)^{1/2} \left( \int_r^{+\infty} |\varphi^\prime(p)|^{2} dp\right)^{1/2} \\
        &\leq \frac{p}{r} \left( \int_r^{+\infty}  |\varphi(p)|^{2p-2} p dp\right)^{1/2} \left( \int_r^{+\infty} |\varphi^\prime(p)|^{2} p dp\right)^{1/2}.
    \end{aligned}
\end{equation}

Hence, for $\varphi \in H^1_{rad}(\mathbb{R}^2)$ and any $x \in \mathbb{R}^2 \setminus \{0\} $ we have
\begin{equation}
    \begin{aligned}
       |\varphi(x)|^p \leq \frac{p}{|x|} \|\varphi\|_{L^{2p-2}}^{p-1} \|\nabla \varphi\|_{L^2} .
    \end{aligned}
\end{equation}

On the other hand,  there exists a constant $C\geq 1,$ so that for any $q\in [2,+\infty)$ and any $\varphi \in H^1(\mathbb{R}^2)$
we have
\begin{equation}
    \begin{aligned}
        \|\varphi\|_{L^q} \leq C q^{1/2} \|\varphi\|_{L^2}^{2/q} \|\nabla \varphi \|_{L^2}^{1-2/q}.
    \end{aligned}
\end{equation}

    }
\begin{equation}
    \begin{aligned}
       & |u(t,x)|^p = |M(-t)u(t,x)|^p \leq \frac{p}{|x|} \|M(-t)u(t)\|_{L^{2p-2}}^{ p-1} \|\nabla M(-t)u(t)\|_{L^{2}}.
    \end{aligned}
\end{equation}
Using the fact that
\begin{equation}
    \begin{aligned}
       &  \|M(-t)u(t)\|_{L^{2p-2}}  \\
        & \leq C (2p-2)^{1/2}  \|M(-t)u(t)\|_{L^{2}}^{2/(2p-2)} \|\nabla M(-t)u(t)\|_{L^{2}}^{1-2/(2p-2)},
    \end{aligned}
\end{equation}
we obtain

\begin{equation}
    \begin{aligned}
        & |u(t,x)|^p \\
        & \leq  \frac{p}{|x|} C^{p-1} 2^{(p-1)/2} p^{(p-1)/2}\|u_0\|_{L^2}  \|\nabla M(-t)u(t)\|_{L^{2}}^{p-1}\\
        & \leq  \frac{p}{|x|} C^{p-1} 2^{(p-1)/2} p^{(p-1)/2}\|u_0\|_{L^2} \left( t^{-1} \|J(t)u(t)\|_{L^{2}}\right)^{p-1}\\
        &\leq \frac{p}{|x| t^{p-1}} C^{p} (2p)^{p/2} \|u_0\|_{L^2}   \|J(t)u(t)\|_{L^{2}}^{p-1}.
    \end{aligned}
\end{equation}
for any $t \geq 1.$
In this way, we find
\begin{equation}
    |u(t,x)| \leq \left(  \frac{p}{|x|} \|u_0\|_{L^2} \right)^{1/p} C (2p)^{1/2} t^{-(p-1)/(2p)}.
\end{equation}
Substitution $\varepsilon = 1/p$ implies \eqref{14}.

\end{proof}

\section{Weak convergence of modified wave operators}

In this section we prove Theorem \ref{th5} that is a generalization of a result in \cite{H95} obtained for the critical Hartree equation in dimensions $n \geq 2.$
Namely, we consider the case $n=2$ and the Cauchy problem

\begin{equation}\label{eq.NLS2a}
    \begin{aligned}
    & (i \partial_t + \frac{1}{2}\Delta_x)u = u |u|, \ \ t \geq 0, x \in \mathbb{R}^2, \\
    & u(0)= u_0.
    \end{aligned}
\end{equation}
As in the work \cite{H95} we set
\begin{equation}
    \Phi_t(u)(x) = \int_1^{+\infty} \left|u\left(\tau, \frac{\tau}{t} x \right)\right| d\tau.
\end{equation}

\begin{proof}[Proof of Theorem \ref{th5}]
  As before, (see \eqref{eq.defvt5} in  Section \ref{s.3})  we set
\begin{equation}\label{eq.defvt5m}
    v(t) = M(t) U(-t) \exp \left( i \Phi_t(u)\right)\  u(t) .
\end{equation}
Lemma \ref{mle60} guarantees that
   \begin{equation}
       \mathcal{F} M(t) U(-t) (g u(t)) = \tilde{g}   \mathcal{F} M(t) U(-t)  u(t) , \ \ \ \tilde{g}(x) = g(x t).
   \end{equation}
   Then we have \eqref{eq.meme0}
   and we can refer to Corollary \ref{cor3.1} that implies
  \begin{equation}\label{eq.c3.1}
    \begin{aligned}
    & i \frac{d}{dt} v(t)  = \mathcal{F}^{-1}  \exp \left(i \int_0^t |u(\tau, \tau x)| d\tau \right) \frac{(-\Delta)}{2t^2}\mathcal{F} M(t) U(-t) u(t).
    \end{aligned}
\end{equation}

For any $t>s\geq 1$ and for any $\varphi \in \mathcal{S}(\mathbb{R}^2)$ we obtain

\begin{equation}
    \begin{aligned}
       &  ( v(t)-v(s)| \varphi)= -i \int_s^t  ( iv^\prime(\tau)| \varphi) d\tau \\
       & =  \frac{i}{2} \int_s^t \tau^{-2} \left(  \mathcal{F}^{-1}  \exp \left(i \tilde{\Phi}_\tau  \right) \Delta \mathcal{F} \mathbb{W}(\tau)| \varphi  \right ) d\tau\\
       & =\frac{i}{2} \int_s^t \tau^{-2} \left(   \exp \left(i \tilde{\Phi}_\tau  \right) \Delta \mathcal{F} \mathbb{W}(\tau)| \mathcal{F}(\varphi)  \right ) d\tau = \\
       & = -\frac{i}{2} \int_s^t \tau^{-2} \left(   \exp \left(i \tilde{\Phi}_\tau  \right) \nabla \mathcal{F} \mathbb{W}(\tau)| \nabla \mathcal{F}(\varphi)  \right ) d\tau  \\
       &- \frac{i}{2} \int_s^t \tau^{-2} \left( \nabla \left(  \exp \left(i \tilde{\Phi}_\tau  \right)\right)  \nabla \mathcal{F} \mathbb{W}(\tau)| \mathcal{F}(\varphi)  \right ) d\tau =: I + II,
    \end{aligned}
\end{equation}
where
\begin{equation}
    \mathbb{W}(t) = M(t)U(-t) u(t)
\end{equation}
and
\begin{equation}\label{tild1m}
    \tilde{\Phi}_t(x) = \int_1^t  |u(\tau, \tau x)| d \tau.
\end{equation}
Further we estimate each of the terms $I$ and $II.$
\begin{equation}
\begin{aligned}
    &|I| \leq \int_s^t \tau^{-2} \|J(\tau)u(\tau)\|_{L^2}  d\tau \| \nabla \mathcal{F}(\varphi)  \|_{L^2}  \\
    & \leq \left( \int_s^t \tau^{-2} \right)^{1/2} \left( \int_s^t \tau^{-2} \| J(\tau)u(\tau)\|^2_{L^2}  d\tau\right)^{1/2}
    \| \nabla \mathcal{F}(\varphi)  \|_{L^2}\\
    & \leq \left( \frac{1}{s}-\frac{1}{t} \right)^{1/2}  \left( \int_s^t \tau^{-2} \| J(\tau)u(\tau)\|^2_{L^2}  d\tau\right)^{1/2}
    \| \nabla \mathcal{F}(\varphi)  \|_{L^2}.
\end{aligned}
\end{equation}
Inequality \eqref{eq.jco29} guarantees the smallness of $I.$

The gradient of the phase factor in $II$ is estimated as follows.
\begin{equation}
    \begin{aligned}
       &\nabla  \exp \left(i \tilde{\Phi}_t  \right)  =  i \exp \left(i \tilde{\Phi}_t  \right)  \nabla \int_1^t \left|u\left(\sigma,\sigma x \right)\right| d\sigma \\
       =&  i \exp \left(i \tilde{\Phi}_t  \right) \mathrm{Re} \int_1^t \sigma \frac{\overline{u\left(\sigma,\sigma x \right)}}{\left| u\left(\sigma,\sigma x \right)\right|} \nabla  u\left(\sigma,\sigma x \right) d\sigma \\
       &= i \exp \left(i \tilde{\Phi}_t  \right) \mathrm{Im} \int_1^t  \frac{\overline{u\left(\sigma,\sigma x \right)}}{\left| u\left(\sigma,\sigma x \right)\right|} J(\sigma)  u\left(\sigma,\sigma x \right) d\sigma
    \end{aligned}
\end{equation}
and we obtain
\begin{equation}
    \begin{aligned}
      &\left\| \nabla  \exp \left(i \tilde{\Phi}_\tau  \right)     \right\|_{L^2}  \leq \int_1^\tau \left\| J(\sigma)  u\left(\sigma, \sigma \cdot\right) \right\|_{L^2} d \sigma \\
      & = \int_1^\tau \sigma^{-1}\left\| J(\sigma)  u\left(\sigma\right) \right\|_{L^2} d \sigma.
    \end{aligned}
\end{equation}

We have also
\begin{equation}
    \begin{aligned}
        &\left\|\nabla \mathcal{F} \mathbb{W}(\tau)\right\|_{L^2} \lesssim  \left\|\nabla \mathcal{F}M(\tau) U(-t)u(\tau)\right\|_{L^2} \lesssim \|J(\tau)u(\tau)\|_{L^2}.
    \end{aligned}
\end{equation}
In this way, we can obtain
\begin{equation}
    \begin{aligned}
     & |II|  = \left| \frac{1}{2} \int_s^t \tau^{-2} \left( \nabla \left(  \exp \left(i \tilde{\Phi}_\tau  \right)\right)  \nabla \mathcal{F} \mathbb{W}(\tau)| \mathcal{F}(\varphi)  \right ) d\tau \right|  \\
      & \lesssim \int_s^t \tau^{-2} \left\| \nabla \left(  \exp \left(i \tilde{\Phi}_\tau  \right)\right)\right\|_{L^2} \left\|\nabla \mathcal{F}\mathbb{W}(\tau)\right\|_{L^2} \| \|\mathcal{F}(\varphi)\|_{L^\infty}   d\tau\\
      & \lesssim \left[\int_s^t \tau^{-2} \left( \int_1^\tau \sigma^{-1}\left\| J(\sigma)  u\left(\sigma\right) \right\|_{L^2} d \sigma \right) \|J(\tau)u(\tau)\|_{L^2} d\tau \right]  \|\mathcal{F}(\varphi)\|_{L^\infty} \\
      & \lesssim \left(\int_s^t \tau^{-2} \left( \int_1^\tau \sigma^{-1}\left\| J(\sigma)  u\left(\sigma\right) \right\|_{L^2} d \sigma \right)^2 d\tau \right)^{1/2} \\
      & \times \left( \int_s^t \tau^{-2}\|J(\tau)u(\tau)\|^2_{L^2} d\tau\right)^{1/2}   \|\mathcal{F}(\varphi)\|_{L^\infty}.
    \end{aligned}
\end{equation}
To  this end, we apply the Hardy inequality \eqref{H1} and deduce
\begin{equation}
    \begin{aligned}
      &  |II| \lesssim \left(\int_s^t \tau^{-2}\|J(\tau)u(\tau)\|^2_{L^2} d\tau\right)^{1/2} H(t)^{1/2} \|\mathcal{F}(\varphi)\|_{L^\infty}\\
        & = (H(t)-H(s))^{1/2} H(t)^{1/2} \|\mathcal{F}(\varphi)\|_{L^\infty},
    \end{aligned}
\end{equation}
where
\begin{equation}
    H(t) = \int_1^t \tau^{-2} \|J(\tau)u(\tau)\|^2_{L^2} d\tau.
\end{equation}
Collecting everything, we obtain
\begin{equation}\label{pest2}
    \begin{aligned}
       & \left| ( v(t)-v(s)| \varphi) \right| \lesssim
        \left( \frac{1}{s}-\frac{1}{t} \right)^{1/2}  \left( H(t)-H(s)\right)^{1/2}
    \| \nabla \mathcal{F}(\varphi)  \|_{L^2}\\
    & +(H(t)-H(s))^{1/2} H(t)^{1/2} \|\mathcal{F}(\varphi)\|_{L^\infty}
    \end{aligned}
\end{equation}
for any $\varphi \in \mathcal{S},$ and for any $t,s$ with $t>s \geq 1.$

It is easy to deduce that for any $\varphi \in L^2$ the sequence of linear functionals on $L^2$
\begin{equation}
    \Lambda(t) (\varphi) = \left( v(t) | \varphi \right)
\end{equation}
has a pointwise limit as $t \to +\infty.$ In fact, any $\varphi \in L^2$ can be approximated by $\varphi_\varepsilon \in \mathcal{S}$ in $L^2$ so that $\|\varphi - \varphi_\varepsilon\|_{L^2} \leq \varepsilon.$
So we have
\begin{equation}
    \begin{aligned}
       & \left| \left( v(t) -v(s) | \varphi \right) \right|\\
        & \leq \left| \left( v(t) -v(s) | \varphi-\varphi_\varepsilon \right) \right|+ \left| \left( v(t) -v(s) | \varphi_\varepsilon \right) \right| \\
        & \leq 2 \|u_0\|_{L^2} \varepsilon + \left| \left( v(t) -v(s) | \varphi_\varepsilon \right) \right|
    \end{aligned}
\end{equation}
so the estimate \eqref{pest2} shows that $\{((v(t)| \varphi ), t >1\}$ is a Cauchy sequence and the functionals $\Lambda(t), t >1$ are uniformly bounded on $L^2$ due to Banach - Steinhaus theorem so there exists $u_+ \in L^2$ so that
\begin{equation}
    \lim_{t\to +\infty} |((v(t)-u_+) | \varphi) | = 0.
\end{equation}
We can exploit \eqref{pest2} to show the convergence in appropriate Banach space.
Turning back to the choice $\varphi \in \mathcal{S},$ we see that \eqref{pest2}  implies also
\begin{equation}\label{pest2mm}
    \begin{aligned}
       & \left| ( v(t)-u_+| \varphi) \right| \lesssim
        \left( \frac{1}{t}\right)^{1/2}  \left( H(t)-H_0\right)^{1/2}
    \| \nabla \mathcal{F}(\varphi)  \|_{L^2}\\
    & +(H(t)-H_0)^{1/2} H(t)^{1/2} \|\mathcal{F}(\varphi)\|_{L^\infty}.
    \end{aligned}
\end{equation}
where
\begin{equation}
    H_0 = \lim_{t\to +\infty} H(t) = \int_1^{+\infty} \tau^{-2} \|J(\tau)u(\tau)\|^2_{L^2} d\tau.
\end{equation}
Now we turn to the limit in \eqref{wp1}.

We take $\varphi \in \mathcal{S}$ and then we have

\begin{equation}
    \begin{aligned}
        &\left|\left(\left[U(-t) \exp \left(i \Phi_t(u) \right) u(t) - u_+\right]| \varphi \right)\right| \\
         = &\left| \left(\left[M(t)U(-t) \exp \left(i \Phi_t(u) \right) u(t) - M(t)u_+\right]| M(t)\varphi \right)\right|\\
         = &\left| \left(v(t) - M(t)u_+| M(t)\varphi \right)\right| \\
        \leq  &  \left| \left(v(t) - M(t)u_+| \varphi \right)\right| +  \left| \left(v(t) - M(t)u_+| M(t)\varphi - \varphi \right)\right|\\
        \leq&  \left| \left(v(t) - u_+| \varphi \right)\right| + \left| \left(u_+ - M(t) u_+| \varphi \right)\right| + \left| \left([v(t) - M(t)u_+| M(t)\varphi - \varphi \right)\right|\\
        &\lesssim \left| \left(v(t) - u_+| \varphi \right)\right| + (\|u_0\|_{L^2}+\|u_+\|_{L^2}) \|  M(t)\varphi - \varphi\|_{L^2}\\
        & \lesssim \left( \frac{1}{t}\right)^{1/2}  \left( H(t)-H_0\right)^{1/2}
    \| \nabla \mathcal{F}(\varphi)  \|_{L^2} +(H(t)-H_0)^{1/2} H(t)^{1/2} \|\mathcal{F}(\varphi)\|_{L^\infty} \\
    &+  \left( \frac{1}{t}\right)^{1/2} (\|u_0\|_{L^2}+\|u_+\|_{L^2})  \| \nabla \mathcal{F}(\varphi)  \|_{L^2}.
    \end{aligned}
\end{equation}

We have continuous embeddings
\begin{equation}
\begin{aligned}
     &\dot{B}^1_{2,1} \hookrightarrow \dot{B}^1_{2,2} = \dot{H}^1, \\
    & \dot{B}^1_{2,1} \hookrightarrow L^\infty
\end{aligned}
\end{equation}
also the duality relation $ ( \dot{B}^1_{2,1})^\prime = \dot{B}^{-1}_{2,{\infty}}$ so we obtain the estimates
\begin{equation}
    \begin{aligned}
       &\left\|U(-t) \exp \left(i \Phi_t(u) \right) u(t) - u_+\right\|_{\mathcal{F}\dot{B}^{-1}_{2,\infty}} \\
        = &\sup_{\varphi \in \mathcal{F}\dot{B}^1_{2,1}, \|\mathcal{F} \varphi\|_{\dot{B}^1_{2,1}}=1  } \left|\left(\left[U(-t) \exp \left(i \Phi_t(u) \right) u(t) - u_+\right]| \varphi \right)\right| \\
        &\lesssim \left(  \frac{1}{\sqrt{t}} +  H(t)^{1/2} \right)   \left( H(t)-H_0\right)^{1/2}
   +  \frac{1}{\sqrt{t}}  (\|u_0\|_{L^2}+\|u_+\|_{L^2}) .
    \end{aligned}
\end{equation}

This completes the proof.

\end{proof}

\begin{remark}
    The spaces $\mathcal{F}\dot{B}^1_{2,1}$ and $\mathcal{F}\dot{B}^{-1}_{2,\infty}$ correspond to the homogeneous versions of $B_1$ and $B_1^*$ in \cite{AH76}, respectively.
\end{remark}

\begin{remark}
    For any Cauchy data in $u_0 \in \mathcal{H}^{1,1}$ the existence of a weak limit
    \begin{equation}
        w-\lim_{t \to {+\infty}} M(t) U(-t) \exp \left(i \Phi_t(u) \right) u(t)
    \end{equation}
    in $L^2$ has been proved by Hirata \cite{H95}.
\end{remark}

\section{Proof of Theorem \ref{th6}}

An important point in the proof of Theorem \ref{th6} is played by the following.
\begin{prop}
    For  any $z \in \mathbb{C}$ we have
    \begin{equation}\label{(1)}
        |f(z)| \leq 2 |z|^{1+2/n},
    \end{equation}
    \begin{equation}\label{(2)}
       0 \leq V(z) \leq 3 |z|^{2+2/n},
    \end{equation}
    \begin{equation}\label{(3)}
       W(z) = V(z) + n((1+|z|^{2})^{1/n}-1)   - |z|^2,
    \end{equation}
    \begin{equation}\label{(4)}
      W(z) \leq V(z),
    \end{equation}
    where $f(z), V(z), W(z)$ are defined in \eqref{fvW}.
\end{prop}

\begin{proof}
    For $\lambda >0,$ we put $\varphi(t) = (1+t)^\lambda -1,$ $t >0.$ If $0 < \lambda < 1,$ then
    \begin{equation}
        0 \leq \varphi(t) = \varphi(t)-\varphi(0) = \int_0^t \varphi(\tau) d\tau = \lambda \int_0^t (1+s)^{\lambda-1} ds \leq \lambda t.
    \end{equation}
    This leads to
    \begin{equation}
        |f(z)| = ((1+|z|^{2})^{1/n}-1) |z| \leq \frac{|z|^3}{n}.
    \end{equation}
    In particular, if $|z| \leq 1$, then
    \begin{equation}\label{(5)}
        |f(z)| \leq \frac{1}{n} |z|^{1+2/n} .
    \end{equation}
    If $|z|\geq 1,$ then we have
    \begin{equation}\label{(6)}
        |f(z)|  \leq (1+|z|^{2})^{1/n} |z| \leq   (2|z|^{2})^{1/n} |z| \leq 2 |z|^{1+2/n} .
    \end{equation}
    By \eqref{(5)}
 and \eqref{(6)}, we have \eqref{(1)}  .

 If $1<\lambda < 2,$
 then
 \begin{equation}
 \begin{aligned}
     & 0 \leq \varphi(t) = \varphi^\prime(0) t + \int_0^t (t-s) \varphi^{\prime\prime}(s)ds \\
     &  = \lambda t + \lambda (\lambda-1) \int_0^t (t-s)(1+s)^{\lambda-2} ds \leq \lambda t + \frac{\lambda(\lambda-1)}{2} t^2.
 \end{aligned}
 \end{equation}
 This implies
 \begin{equation}\label{(7)}
     \begin{aligned}
      &   0 \leq V(z) = \frac{n}{n+1} \left[(1+|z|^{2})^{(n+1)/n}-1)- \frac{n+1}{n}z^2 \right] \\
      & \leq \frac{n}{n+1}\frac{n+1}{n} \frac{1}{2n}|z|^4 = \frac{1}{2n}|z|^4.
     \end{aligned}
 \end{equation}
In particular, if $|z|\leq 1,$
\begin{equation}
    0 \leq V(z) \leq \frac{1}{2n}|z|^{2+2/n}.
\end{equation}
If $|z| \geq 1, $ then we have
\begin{equation}\label{(8)}
    \begin{aligned}
        &0 \leq V(z) \leq \frac{n}{n+1} (1+|z|^{2})^{(n+1)/n} \\
        &\leq \frac{n}{n+1} (2|z|^{2})^{(n+1)/n} \leq \frac{3n}{n+1} |z|^{2+2/n} .
    \end{aligned}
\end{equation}
By \eqref{(7)} and \eqref{(8)} we have \eqref{(2)}.

Using the relation
$$ \langle z \rangle^{2+2/n} - 1 = |z|^2\langle z \rangle^{2/n} + \langle z \rangle^{2/n}-1,$$ we find
\begin{equation}
    \begin{aligned}
       & (n+1)V(z) -n \overline{z}f(z) = (n+1) \left[ \frac{n}{n+1} (\langle z \rangle^{2+2/n}-1)   - |z|^2 \right] \\
       &-n |z|^2 \left[(\langle z \rangle^{2/n} -1) \right]\\
        & = \frac{(n+1)n}{n+1} \left(\langle z \rangle^{2/n}-1 \right) + \frac{(n+1)n}{n+1}|z|^2 \langle z \rangle^{2/n} \\
        &-(n+1)|z|^2 -n |z|^2  (\langle z \rangle^{2/n}-1)\\
        & = n\left(\langle z \rangle^{2/n}-1 \right) -|z|^2 .
\end{aligned}
\end{equation}
Hence   we can write
\begin{equation}
    \begin{aligned}
        & W(z) = (n+2) V(z) -n \overline{z} f(z)= V(z) + (n+1)V(z)- n \overline{z} f(z)\\
        &= V(z) + n\left(\langle z \rangle^{2/n}-1 \right) -|z|^2
    \end{aligned}
\end{equation}
 which shows \eqref{(3)}.

For $\lambda \in (0,1)$ we have
\begin{equation}
    (1+t)^\lambda -1-\lambda t = \lambda (\lambda-1)\int_0^t(t-s)(1+s)^{\lambda-2} ds \leq 0.
\end{equation}
 This leads to
 \begin{equation}
     n\left( \langle z \rangle^{2/n}-1\right) -|z|^2 \leq 0
 \end{equation}
 and \eqref{(4)}.
    \end{proof}

Let $u_0\in L^2(\mathbb{R}^n).$ Then there exists a unique global solution $u \in C(\mathbb{R};L^2)$ in Strichartz class, such that
\begin{equation}
   \|u(t)\|_{L^2} = \|u_0\|_{L^2}
\end{equation}
for any $t \in \mathbb{R}.$

Let $u_0\in H^1(\mathbb{R}^n).$ Then there exists a unique global solution $u \in C(\mathbb{R};H^1)$ in Strichartz class, such that
\begin{equation}
    E(t) =E(0)
\end{equation}
for any $t \in \mathbb{R},$ where
\begin{equation}
    E(t) = \frac{1}{2} \|\nabla u(t)\|_{L^2}^2 + \int_{\mathbb{R}^n} V(u(t,x)) dx .
\end{equation}
and $V(z)$ defined in \eqref{fvW}.

Let $u_0\in L^2(\mathbb{R}^n)$ satisfy $xu_0 \in L^2. $ Then the global $L^2 -$ solution $u$ satisfies
$xU(-t)u(t) \in L^2$ and
\begin{equation}
    \|J(t)u(t)\|_{L^2}^2 + t^2 \int_{\mathbb{R}^n} V(u(t,x)) dx = \|xu_0\|_{L^2}^2 +  \int_{\mathbb{R}^n} W(u(t,x)) dx
\end{equation}
for any $t \in \mathbb{R}$ and $V,W$ defined in \eqref{fvW}.

Moreover,
\begin{equation}
    \begin{aligned}
        &\frac{d}{dt} \left(  t^{-1}\left( \|J(t)u(t)\|_{L^2}^2 + t^2 \int_{\mathbb{R}^n} V(u(t,x)) dx\right) \right) \\
        &=-t^{-2} \left( \|J(t)u(t)\|_{L^2}^2 + t^2 \int_{\mathbb{R}^n} V(u(t,x)) dx \right) + \int_{\mathbb{R}^n} W(u(t,x)) dx \\
        &= -t^{-2} \|J(t)u(t)\|_{L^2}^2 + \int_{\mathbb{R}^n} W(u(t,x)) dx - \int_{\mathbb{R}^n} V(u(t,x)) dx\\
        & \leq -t^{-2} \|J(t)u(t)\|_{L^2}^2.
    \end{aligned}
\end{equation}
This yields
\begin{equation}
    \begin{aligned}
         &t^{-1} \|J(t)u(t)\|_{L^2}^2 + t \int_{\mathbb{R}^n} V(u(t,x)) dx + \int_1^t s^{-2} \|J(s)u(s)\|_{L^2}^2 ds\\
         & \leq \|J(1)u(1)\|_{L^2}^2 + \int_{\mathbb{R}^n} V(u(1,x)) dx
    \end{aligned}
\end{equation}
for any $t \geq 1.$

We define the phase function
\begin{equation}
    \Phi_t(u)(x) = \int_1^{+\infty} \left[ (1+|u(\tau, \tau x/t)|^2)^{1/n}-1 \right] d\tau.
\end{equation}

\begin{proof}[Proof of Theorem \ref{th6}]
  We set
\begin{equation}\label{eq.defvt6m}
    v(t) = M(t) U(-t) \exp \left( i \Phi_t(u)\right)\  u(t) .
\end{equation}
Lemma \ref{mle60} guarantees that
   \begin{equation}
       \mathcal{F} M(t) U(-t) (g u(t)) = \tilde{g}   \mathcal{F} M(t) U(-t)  u(t) , \ \ \ \tilde{g}(x) = g(x t).
   \end{equation}
Therefore, the Dollard decomposition implies
\begin{equation}
    v(t) = \mathcal{F}^{-1} \mathcal{F} M(t) U(-t) e^{i \Phi_t(u)} u(t)  =  \mathcal{F}^{-1} e^{i \tilde{\Phi}_t} \mathcal{F} M(t) U(-t)  u(t) ,
\end{equation}
where
\begin{equation}\label{tphi}
    \tilde{\Phi}_t(x) = \int_1^t \left[ (1+|u(\tau, \tau x)|^2)^{1/n}-1 \right]d \tau.
\end{equation}

   Then we have \eqref{eq.meme0}
   and we can refer to Corollary \ref{cor3.1} that implies
  \begin{equation}\label{eq.c3.1m}
    \begin{aligned}
    & i \frac{d}{dt} v(t)  = \mathcal{F}^{-1}  \exp \left(i \tilde{\Phi}_t(x) \right) \frac{(-\Delta)}{2t^2}\mathcal{F} M(t) U(-t) u(t).
    \end{aligned}
\end{equation}

For any $t>s\geq 1$ and for any $\varphi \in \mathcal{S}(\mathbb{R}^n)$ we obtain

\begin{equation}
    \begin{aligned}
       &  ( v(t)-v(s)| \varphi)= -i \int_s^t  ( iv^\prime(\tau)| \varphi) d\tau \\
       & =  \frac{i}{2} \int_s^t \tau^{-2} \left(  \mathcal{F}^{-1}  \exp \left(i \tilde{\Phi}_\tau  \right) \Delta \mathcal{F} \mathbb{W}(\tau)| \varphi  \right ) d\tau\\
       & =\frac{i}{2} \int_s^t \tau^{-2} \left(   \exp \left(i \tilde{\Phi}_\tau  \right) \Delta \mathcal{F} \mathbb{W}(\tau)| \mathcal{F}(\varphi)  \right ) d\tau  \\
       & = -\frac{i}{2} \int_s^t \tau^{-2} \left(   \exp \left(i \tilde{\Phi}_\tau  \right) \nabla \mathcal{F} \mathbb{W}(\tau)| \nabla \mathcal{F}(\varphi)  \right ) d\tau  \\
       &- \frac{i}{2} \int_s^t \tau^{-2} \left( \nabla \left(  \exp \left(i \tilde{\Phi}_\tau  \right)\right)  \nabla \mathcal{F} \mathbb{W}(\tau)| \mathcal{F}(\varphi)  \right ) d\tau =: I + II,
    \end{aligned}
\end{equation}
where
\begin{equation}
    \mathbb{W}(t) = M(t)U(-t) u(t).
\end{equation}
Further we estimate each of the terms $I$ and $II.$
\begin{equation}
\begin{aligned}
    &|I| \leq \int_s^t \tau^{-2} \|J(\tau)u(\tau)\|_{L^2}  d\tau \| \nabla \mathcal{F}(\varphi)  \|_{L^2}  \\
    & \leq \left( \int_s^t \tau^{-2} \right)^{1/2} \left( \int_s^t \tau^{-2} \| J(\tau)u(\tau)\|^2_{L^2}  d\tau\right)^{1/2}
    \| \nabla \mathcal{F}(\varphi)  \|_{L^2}\\
    & \leq \left( \frac{1}{s}-\frac{1}{t} \right)^{1/2}  \left( \int_s^t \tau^{-2} \| J(\tau)u(\tau)\|^2_{L^2}  d\tau\right)^{1/2}
    \| \nabla \mathcal{F}(\varphi)  \|_{L^2}.
\end{aligned}
\end{equation}
Inequality \eqref{eq.jco29} guarantees the smallness of $I.$

The gradient of the phase factor in $II$ is estimated as follows.
\begin{equation}
    \begin{aligned}
       &\nabla  \exp \left(i \tilde{\Phi}_t  \right)  =  i \exp \left(i \tilde{\Phi}_t  \right)  \nabla \tilde{\Phi}_t \\
       =&  i \frac{2}{n}\exp \left(i \tilde{\Phi}_t  \right) \int_1^t \sigma \left( u(\sigma, \sigma x )\right)^{2/n-2} \ \mathrm{Re} \left[ \overline{u\left(\sigma,\sigma x \right)} \nabla  u\left(\sigma,\sigma x \right)\right] d\sigma \\
       &= i \frac{2}{n}\exp \left(i \tilde{\Phi}_t  \right) \int_1^t \sigma \left( u(\sigma, \sigma x )\right)^{2/n-2} \ \mathrm{Im} \left[ \overline{u\left(\sigma,\sigma x \right)} Ju\left(\sigma,\sigma x \right)\right] d\sigma
    \end{aligned}
\end{equation}
and is estimated as
\begin{equation}
    \begin{aligned}
      &\left\| \nabla  \exp \left(i \tilde{\Phi}_\tau  \right)     \right\|_{L^2}  \leq \frac{2}{n}\int_1^\tau \left\| J(\sigma)  u\left(\sigma, \sigma \cdot\right) \right\|_{L^2} d \sigma \\
      & = \frac{2}{n}\int_1^\tau \sigma^{-n/2}\left\| J(\sigma)  u\left(\sigma\right) \right\|_{L^2} d \sigma \\
      & \lesssim  \int_1^\tau \sigma^{-1}\left\| J(\sigma)  u\left(\sigma\right) \right\|_{L^2} d \sigma=: f(\tau).
    \end{aligned}
\end{equation}

We have also
\begin{equation}
    \begin{aligned}
        &\left\|\nabla \mathcal{F} \mathbb{W}(\tau)\right\|_{L^2} \lesssim  \left\|\nabla \mathcal{F}M(\tau) U(-t)u(\tau)\right\|_{L^2} \lesssim \|J(\tau)u(\tau)\|_{L^2}
    \end{aligned}
\end{equation}
In this way, we can obtain
\begin{equation}
    \begin{aligned}
     & |II|  = \left| \frac{1}{2} \int_s^t \tau^{-2} \left( \nabla \left(  \exp \left(i \tilde{\Phi}_\tau  \right)\right)  \nabla \mathcal{F} \mathbb{W}(\tau)| \mathcal{F}(\varphi)  \right ) d\tau \right|  \\
      & \lesssim \int_s^t \tau^{-2} \left\| \nabla \left(  \exp \left(i \tilde{\Phi}_\tau  \right)\right)\right\|_{L^2} \left\|\nabla \mathcal{F}\mathbb{W}(\tau)\right\|_{L^2} \| \|\mathcal{F}(\varphi)\|_{L^\infty}   d\tau\\
      & \lesssim \left[\int_s^t \tau^{-2} \left( \int_1^\tau \sigma^{-1}\left\| J(\sigma)  u\left(\sigma\right) \right\|_{L^2} d \sigma \right) \|J(\tau)u(\tau)\|_{L^2} d\tau \right]  \|\mathcal{F}(\varphi)\|_{L^\infty} \\
      & \lesssim \left(\int_s^t \tau^{-2} \left( \int_1^\tau \sigma^{-1}\left\| J(\sigma)  u\left(\sigma\right) \right\|_{L^2} d \sigma \right)^2 d\tau \right)^{1/2}  \\
      & \times \left( \int_s^t \tau^{-2}\|J(\tau)u(\tau)\|^2_{L^2} d\tau\right)^{1/2}   \|\mathcal{F}(\varphi)\|_{L^\infty}.
    \end{aligned}
\end{equation}
To  this end, we apply the Hardy inequality \eqref{H1} and deduce
\begin{equation}
    \begin{aligned}
      &  |II| \lesssim \left(\int_s^t \tau^{-2}\|J(\tau)u(\tau)\|^2_{L^2} d\tau\right)^{1/2} H(t)^{1/2} \|\mathcal{F}(\varphi)\|_{L^\infty}\\
        & = (H(t)-H(s))^{1/2} H(t)^{1/2} \|\mathcal{F}(\varphi)\|_{L^\infty},
    \end{aligned}
\end{equation}
where
\begin{equation}
    H(t) = \int_1^t \tau^{-2} \|J(\tau)u(\tau)\|^2_{L^2} d\tau.
\end{equation}
Collecting everything, we obtain
\begin{equation}\label{pest2oo}
    \begin{aligned}
       & \left| ( v(t)-v(s)| \varphi) \right| \lesssim
        \left( \frac{1}{s}-\frac{1}{t} \right)^{1/2}  \left( H(t)-H(s)\right)^{1/2}
    \| \nabla \mathcal{F}(\varphi)  \|_{L^2}\\
    & +(H(t)-H(s))^{1/2} H(t)^{1/2} \|\mathcal{F}(\varphi)\|_{L^\infty}.
    \end{aligned}
\end{equation}
for any $\varphi \in \mathcal{S},$ and for any $t,s$ with $t>s \geq 1.$

It is easy to deduce that for any $\varphi \in L^2$ the sequence of linear functionals on $L^2$
\begin{equation}
    \Lambda(t) (\varphi) = \left( v(t) | \varphi \right)
\end{equation}
has a pointwise limit as $t \to {+\infty}.$ In fact, any $\varphi \in L^2$ can be approximated by $\varphi_\varepsilon \in \mathcal{S}$ in $L^2$ so that $\|\varphi - \varphi_\varepsilon\|_{L^2} \leq \varepsilon.$
So we have
\begin{equation}
    \begin{aligned}
       & \left| \left( v(t) -v(s) | \varphi \right) \right|\\
        & \leq \left| \left( v(t) -v(s) | \varphi-\varphi_\varepsilon \right) \right|+ \left| \left( v(t) -v(s) | \varphi_\varepsilon \right) \right| \\
        & \leq 2 \|u_0\|_{L^2}\  \varepsilon + \left| \left( v(t) -v(s) | \varphi_\varepsilon \right) \right|
    \end{aligned}
\end{equation}
and the estimate \eqref{pest2} shows that $\{((v(t)| \varphi ), t >1\}$ is a Cauchy sequence and the functionals $\Lambda(t), t >1$ are uniformly bounded on $L^2.$ From Banach - Steinhaus theorem we see that there exists $u_+ \in L^2$ so that
\begin{equation}
    \lim_{t\to {+\infty}} |((v(t)-u_+), \varphi) | = 0.
\end{equation}
We can exploit \eqref{pest2} to show the convergence in appropriate Banach space.
Turning back to the choice $\varphi \in \mathcal{S},$ we see that \eqref{pest2}  implies also
\begin{equation}\label{pest2mmm}
    \begin{aligned}
       & \left| ( v(t)-u_+| \varphi) \right| \lesssim
        \left( \frac{1}{t}\right)^{1/2}  \left( H(t)-H_0\right)^{1/2}
    \| \nabla \mathcal{F}(\varphi)  \|_{L^2}\\
    & +(H(t)-H_0)^{1/2} H(t)^{1/2} \|\mathcal{F}(\varphi)\|_{L^\infty},
    \end{aligned}
\end{equation}
where
\begin{equation}
    H_0 = \lim_{t\to {+\infty}} H(t) = \int_1^{+\infty} \tau^{-2} \|Ju(\tau)\|^2_{L^2} d\tau.
\end{equation}
Now we turn to the limit in \eqref{wp1}.

We take $\varphi \in \mathcal{S}$ and then we have

\begin{equation}
    \begin{aligned}
        &\left|\left(\left[U(-t) \exp \left(i \Phi_t(u) \right) u(t) - u_+\right]| \varphi \right)\right| \\
         = &\left| \left(\left[M(t)U(-t) \exp \left(i \Phi_t(u) \right) u(t) - M(t)u_+\right]| M(t)\varphi \right)\right|\\
         = &\left| \left(v(t) - M(t)u_+| M(t)\varphi \right)\right| \\
        \leq  &  \left| \left(v(t) - M(t)u_+| \varphi \right)\right| +  \left| \left(v(t) - M(t)u_+| M(t)\varphi - \varphi \right)\right|\\
        \leq&  \left| \left(v(t) - u_+| \varphi \right)\right| + \left| \left(u_+ - M(t) u_+| \varphi \right)\right| + \left| \left([v(t) - M(t)u_+| M(t)\varphi - \varphi \right)\right|\\
        &\lesssim \left| \left(v(t) - u_+| \varphi \right)\right| + (\|u_0\|_{L^2}+\|u_+\|_{L^2}) \|  M(t)\varphi - \varphi\|_{L^2}\\
        & \lesssim \left( \frac{1}{t}\right)^{1/2}  \left( H(t)-H_0\right)^{1/2}
    \| \nabla \mathcal{F}(\varphi)  \|_{L^2} +(H(t)-H_0)^{1/2} H(t)^{1/2} \|\mathcal{F}(\varphi)\|_{L^\infty} \\
    &+  \left( \frac{1}{t}\right)^{1/2} (\|u_0\|_{L^2}+\|u_+\|_{L^2})  \| \nabla \mathcal{F}(\varphi)  \|_{L^2}.
    \end{aligned}
\end{equation}

We have continuous embeddings
\begin{equation}
\begin{aligned}
     &\dot{B}^{n/2}_{2,1} \hookrightarrow \dot{B}^1_{2,2} = \dot{H}^1, \\
    & \dot{B}^{n/2}_{2,1} \hookrightarrow L^\infty
\end{aligned}
\end{equation}
also the duality relation $ ( \dot{B}^{n/2}_{2,1})^\prime = \dot{B}^{-n/2}_{2,\infty}$ so we obtain the estimates
\begin{equation}
    \begin{aligned}
       &\left\|U(-t) \exp \left(i \Phi_t(u) \right) u(t) - u_+ \right\|_{\mathcal{F}\dot{B}^{-n/2}_{2,\infty}} \\
        = &\sup_{\varphi \in \mathcal{F}\dot{B}^{n/2}_{2,1}, \|\mathcal{F} \varphi\|_{\dot{B}^{n/2}_{2,1}}=1  } \left|\left(\left[U(-t) \exp \left(i \Phi_t(u) \right) u(t) - u_+\right]| \varphi \right)\right| \\
        &\lesssim \left(  \frac{1}{\sqrt{t}} +  H(t)^{1/2} \right)   \left( H(t)-H_0\right)^{1/2}
   +  \frac{1}{\sqrt{t}}  (\|u_0\|_{L^2}+\|u_+\|_{L^2}) .
    \end{aligned}
\end{equation}

This completes the proof.

\end{proof}


\section*{Acknowledgements}

The authors are grateful to J.Murphy, H.Mizutani and  H.Miyazaki for the discussions and for pointing out a mistake in the previous version of the manuscript.

The first author was supported in part by
  Gruppo Nazionale per l'Analisi Matematica,  by Institute of Mathematics and Informatics, Bulgarian Academy of Sciences and by Top Global University Project, Waseda University. The second author was supported by JSPS KAKENHI Grant Number 24H00024.

\bibliographystyle{plain}


\bibliography{ModWaveR}

\begin{thebibliography}{10}

\bibitem{AH76}
S.~Agmon and L.~H\"{o}rmander.
\newblock Asymptotic properties of solutions of differential equations with
  simple characteristics.
\newblock {\em J. Analyse Math.}, 30:1--38, 1976.

\bibitem{BL76}
J\"{o}ran Bergh and J\"{o}rgen L\"{o}fstr\"{o}m.
\newblock {\em Interpolation spaces. {A}n introduction}.
\newblock Grundlehren der Mathematischen Wissenschaften, No. 223.
  Springer-Verlag, Berlin-New York, 1976.

\bibitem{C2001}
R\'{e}mi Carles.
\newblock Geometric optics and long range scattering for one-dimensional
  nonlinear {S}chr\"{o}dinger equations.
\newblock {\em Comm. Math. Phys.}, 220(1):41--67, 2001.

\bibitem{C24}
R\'{e}mi Carles.
\newblock Dynamics near the origin of the long range scattering for the
  one-dimensional {S}chr\"{o}dinger equation.
\newblock {\em C. R. Math. Acad. Sci. Paris}, 362:1717--1742, 2024.

\bibitem{C03}
Thierry Cazenave.
\newblock {\em Semilinear {S}chr\"{o}dinger equations}, volume~10 of {\em
  Courant Lecture Notes in Mathematics}.
\newblock New York University, Courant Institute of Mathematical Sciences, New
  York; American Mathematical Society, Providence, RI, 2003.

\bibitem{CW92}
Thierry Cazenave and Fred~B. Weissler.
\newblock Rapidly decaying solutions of the nonlinear {S}chr\"{o}dinger
  equation.
\newblock {\em Comm. Math. Phys.}, 147(1):75--100, 1992.

\bibitem{CGT09}
J.~Colliander, M.~Grillakis, and N.~Tzirakis.
\newblock Tensor products and correlation estimates with applications to
  nonlinear {S}chr\"{o}dinger equations.
\newblock {\em Comm. Pure Appl. Math.}, 62(7):920--968, 2009.

\bibitem{DZ03}
Percy Deift and Xin Zhou.
\newblock Long-time asymptotics for solutions of the {NLS} equation with
  initial data in a weighted {S}obolev space.
\newblock {\em Comm. Pure Appl. Math.}, 56(8):1029--1077, 2003.
\newblock Dedicated to the memory of J\"{u}rgen K. Moser.

\bibitem{FM17}
Kazumasa Fujiwara and Hayato Miyazaki.
\newblock The derivation of conservation laws for nonlinear {S}chr\"{o}dinger
  equations with power type nonlinearities.
\newblock In {\em Regularity and singularity for partial differential equations
  with conservation laws}, RIMS K\^{o}ky\^{u}roku Bessatsu, B63, pages 13--21.
  Res. Inst. Math. Sci. (RIMS), Kyoto, 2017.

\bibitem{GO93}
J.~Ginibre and T.~Ozawa.
\newblock Long range scattering for nonlinear {S}chr\"{o}dinger and {H}artree
  equations in space dimension {$n\geq 2$}.
\newblock {\em Comm. Math. Phys.}, 151(3):619--645, 1993.

\bibitem{GV1980}
J.~Ginibre and G.~Velo.
\newblock Sur une \'{e}quation de {S}chr\"{o}dinger non lin\'{e}aire avec
  interaction non locale.
\newblock In {\em Nonlinear partial differential equations and their
  applications. {C}oll\`ege de {F}rance {S}eminar, {V}ol. {II} ({P}aris,
  1979/1980)}, volume~60 of {\em Res. Notes in Math.}, pages 155--199,
  391--392. Pitman, Boston, Mass.-London, 1982.

\bibitem{H52}
G.~H. Hardy, J.~E. Littlewood, and G.~P\'{o}lya.
\newblock {\em Inequalities}.
\newblock Cambridge Mathematical Library. Cambridge University Press,
  Cambridge, 1988.
\newblock Reprint of the 1952 edition.

\bibitem{HKN98}
Nakao Hayashi, Elena~I. Kaikina, and Pavel~I. Naumkin.
\newblock On the scattering theory for the cubic nonlinear {S}chr\"{o}dinger
  and {H}artree type equations in one space dimension.
\newblock {\em Hokkaido Math. J.}, 27(3):651--667, 1998.

\bibitem{HN98AJM}
Nakao Hayashi and Pavel~I. Naumkin.
\newblock Asymptotics for large time of solutions to the nonlinear
  {S}chr\"{o}dinger and {H}artree equations.
\newblock {\em Amer. J. Math.}, 120(2):369--389, 1998.

\bibitem{HN06}
Nakao Hayashi and Pavel~I. Naumkin.
\newblock Domain and range of the modified wave operator for {S}chr\"{o}dinger
  equations with a critical nonlinearity.
\newblock {\em Comm. Math. Phys.}, 267(2):477--492, 2006.

\bibitem{HN09}
Nakao Hayashi and Pavel~I. Naumkin.
\newblock Asymptotics of odd solutions for cubic nonlinear {S}chr\"{o}dinger
  equations.
\newblock {\em J. Differential Equations}, 246(4):1703--1722, 2009.

\bibitem{HT85}
Nakao Hayashi and Masayoshi Tsutsumi.
\newblock {$L^\infty({\bf R}^n)$}-decay of classical solutions for nonlinear
  {S}chr\"{o}dinger equations.
\newblock {\em Proc. Roy. Soc. Edinburgh Sect. A}, 104(3-4):309--327, 1986.

\bibitem{HT86}
Nakao Hayashi and Yoshio Tsutsumi.
\newblock Remarks on the scattering problem for nonlinear {S}chr\"{o}dinger
  equations.
\newblock In {\em Differential equations and mathematical physics
  ({B}irmingham, {A}la., 1986)}, volume 1285 of {\em Lecture Notes in Math.},
  pages 162--168. Springer, Berlin, 1987.

\bibitem{H95}
Hitoshi Hirata.
\newblock Large time behavior of solution for {H}artree equation with long
  range interaction.
\newblock {\em Tokyo J. Math.}, 18(1):167--177, 1995.

\bibitem{IT24}
Mihaela Ifrim and Daniel Tataru.
\newblock Testing by wave packets and modified scattering in nonlinear
  dispersive {PDE}'s.
\newblock {\em Trans. Amer. Math. Soc. Ser. B}, 11:164--214, 2024.

\bibitem{KP11}
Jun Kato and Fabio Pusateri.
\newblock A new proof of long-range scattering for critical nonlinear
  {S}chr\"{o}dinger equations.
\newblock {\em Differential Integral Equations}, 24(9-10):923--940, 2011.

\bibitem{KM23}
Masaki Kawamoto and Haruya Mizutani.
\newblock Modified scattering for nonlinear schrödinger equations with
  long-range potentials.
\newblock {\em arXiv preprint arXiv:2308.13254}, 2023.

\bibitem{KM25}
Masaki Kawamoto and Haruya Mizutani.
\newblock Modified wave operators for the defocusing cubic nonlinear
  schrödinger equation in one space dimension with large scattering data.
\newblock {\em arXiv preprint arXiv:2506.01871}, 2025.

\bibitem{KO05}
Naoyasu Kita and Tohru Ozawa.
\newblock Sharp asymptotic behavior of solutions to nonlinear {S}chr\"{o}dinger
  equations with repulsive interactions.
\newblock {\em Commun. Contemp. Math.}, 7(2):167--176, 2005.

\bibitem{LS06}
Hans Lindblad and Avy Soffer.
\newblock Scattering and small data completeness for the critical nonlinear
  {S}chr\"{o}dinger equation.
\newblock {\em Nonlinearity}, 19(2):345--353, 2006.

\bibitem{MS16}
G.~Morchio and F.~Strocchi.
\newblock Dynamics of {D}ollard asymptotic variables. {A}symptotic fields in
  {C}oulomb scattering.
\newblock {\em Rev. Math. Phys.}, 28(1):1650001, 26, 2016.

\bibitem{MVH22}
Jason Murphy and Tim Van~Hoose.
\newblock Modified scattering for a dispersion-managed nonlinear
  {S}chr\"{o}dinger equation.
\newblock {\em Nonlinear Differential Equations Appl.}, 29(1):Paper No. 1, 11,
  2022.

\bibitem{NO01}
K.~Nakanishi and T.~Ozawa.
\newblock Scattering problem for nonlinear {S}chr\"{o}dinger and {H}artree
  equations.
\newblock In {\em Nonlinear Differential Equations and Applications}, number
  1234 in Nonlinear Differential Equations and Applications, pages 105--112.
  Springer, 2001.
\newblock Tosio Kato's method and principle for evolution equations in
  mathematical physics (Sapporo, 2001).

\bibitem{NONO02}
Kenji Nakanishi and Tohru Ozawa.
\newblock Remarks on scattering for nonlinear {S}chr\"{o}dinger equations.
\newblock {\em NoDEA Nonlinear Differential Equations Appl.}, 9(1):45--68,
  2002.

\bibitem{O06}
T.~Ozawa.
\newblock Remarks on proofs of conservation laws for nonlinear
  {S}chr\"{o}dinger equations.
\newblock {\em Calc. Var. Partial Differential Equations}, 25(3):403--408,
  2006.

\bibitem{O91}
Tohru Ozawa.
\newblock Long range scattering for nonlinear {S}chr\"{o}dinger equations in
  one space dimension.
\newblock {\em Comm. Math. Phys.}, 139(3):479--493, 1991.

\bibitem{PV09}
Fabrice Planchon and Luis Vega.
\newblock Bilinear virial identities and applications.
\newblock {\em Ann. Sci. \'{E}c. Norm. Sup\'{e}r. (4)}, 42(2):261--290, 2009.

\bibitem{T85}
Yoshio Tsutsumi.
\newblock Scattering problem for nonlinear {S}chr\"{o}dinger equations.
\newblock {\em Ann. Inst. H. Poincar\'{e} Phys. Th\'{e}or.}, 43(3):321--347,
  1985.

\bibitem{TY84}
Yoshio Tsutsumi and Kenji Yajima.
\newblock The asymptotic behavior of nonlinear {S}chr\"{o}dinger equations.
\newblock {\em Bull. Amer. Math. Soc. (N.S.)}, 11(1):186--188, 1984.

\bibitem{W85}
Michael~I. Weinstein.
\newblock Modulational stability of ground states of nonlinear
  {S}chr\"{o}dinger equations.
\newblock {\em SIAM J. Math. Anal.}, 16(3):472--491, 1985.

\end{thebibliography}

\end{document}